\documentclass[12pt]{article}%
\usepackage{url}

\usepackage{color}
\usepackage{amssymb}
\usepackage{amsfonts}
\usepackage{amsmath}
\usepackage[nohead]{geometry}
\usepackage[singlespacing]{setspace}
\usepackage[bottom]{footmisc}
\usepackage{indentfirst}
\usepackage{endnotes}
\usepackage{graphicx}%
\usepackage{rotating}
\usepackage{verbatim}
\usepackage{setspace}
\usepackage{multirow}
\usepackage{latexsym}
\usepackage{bbm}
\usepackage{enumitem}
\usepackage{hhline}
\RequirePackage[colorlinks,citecolor=blue,urlcolor=blue,linkcolor=blue]{hyperref}

\usepackage{graphicx}
\usepackage{caption}
\usepackage{subcaption}

\usepackage{amsthm}
\usepackage{makeidx}
\usepackage{fancyhdr}
\usepackage{type1cm}

 \usepackage[authoryear]{natbib}

\newcommand{\diag}{\mathrm{diag}}
\newcommand{\argmin}{\mathrm{argmin}}

\def \diag {\mbox{diag}}

\def \sgn {\mbox{sgn}}

\newcommand{\beq}{\begin{equation}}
\newcommand{\eeq}{\end{equation}}
\newcommand{\beqn}{\begin{eqnarray}}
\newcommand{\eeqn}{\end{eqnarray}}
\newcommand{\beqnn}{\begin{eqnarray*}}
\newcommand{\eeqnn}{\end{eqnarray*}}

\theoremstyle{plain}
\newtheorem{thm}{Theorem}[section]
\newtheorem{defn}{Definition}[section]
\newtheorem{lem}{Lemma}[section]

\newtheorem{prop}{Proposition}[section]

\theoremstyle{definition}

\include{preamble}

\makeatletter
\def\@biblabel#1{\hspace*{-\labelsep}}
\makeatother
\begin{document}

\title{An $\ell_{\infty}$ Eigenvector Perturbation Bound and Its Application to Robust Covariance Estimation}
\author{Jianqing Fan\thanks{Address: Department of ORFE, Sherrerd Hall, Princeton University, Princeton, NJ 08544, USA, e-mail: \textit{jqfan@princeton.edu}, \textit{yiqiaoz@princeton.edu},
\textit{weichenw@princeton.edu}. The research was partially supported by NSF grants
DMS-1206464 and DMS-1406266 and NIH grants R01-GM072611-10.}\;, Weichen Wang and Yiqiao Zhong
\medskip\\{\normalsize Department of Operations Research and Financial Engineering,  Princeton University}
}

\date{}

\maketitle

\sloppy

\onehalfspacing

\begin{abstract}
In statistics and machine learning, we are  interested in the eigenvectors (or singular vectors) of certain matrices (e.g.\ covariance matrices, data matrices, etc). However, those matrices are usually perturbed by noises or statistical errors, either from random sampling or structural patterns. The Davis-Kahan $\sin \theta$ theorem is often used to bound the difference between the eigenvectors of a matrix $A$ and those of a perturbed matrix $\widetilde{A} = A + E$, in terms of $\ell_2$ norm. In this paper, we prove that when $A$ is a low-rank and incoherent matrix, the $\ell_{\infty}$ norm perturbation bound of singular vectors (or eigenvectors in the symmetric case) is smaller by a factor of $\sqrt{d_1}$ or $\sqrt{d_2}$ for left and right vectors, where $d_1$ and $d_2$ are the matrix dimensions. The power of this new perturbation result is shown in robust covariance estimation, particularly when random variables have heavy tails. There, we propose new robust covariance estimators and establish their asymptotic properties using the newly developed perturbation bound. Our theoretical results are verified through extensive numerical experiments.
\end{abstract}

\textbf{Keywords:} Matrix perturbation theory, Incoherence, Low-rank matrices, Sparsity, Approximate factor model.

\pagebreak%
\doublespacing

\onehalfspacing

\section{Introduction}\label{sec::intro}
The perturbation of matrix eigenvectors (or singular vectors) has been well studied in matrix perturbation theory \citep{Wed72, Ste90}. The best known result of eigenvector perturbation is the classic Davis-Kahan theorem \citep{DavKah70}. It originally emerged as a powerful tool in numerical analysis, but soon found its widespread use in other fields, such as statistics and machine learning. Its popularity continues to surge in recent years, which is largely attributed to the omnipresent data analysis, where it is a common practice, for example, to employ PCA \citep{Jol02} for dimension reduction, feature extraction, and data visualization.

The eigenvectors of matrices are closely related to the underlying structure in a variety of problems. For instance, principal components often capture most information of data and extract the latent factors that drive the correlation structure of the data \citep{Kno99}; in classical multidimensional scaling (MDS), the centered squared distance matrix encodes the coordinates of data points embedded in a low dimensional subspace \citep{BorGro05}; and in clustering and network analysis, spectral algorithms are used to reveal clusters and community structure \citep{Ng02, Roh11}. In those problems, the low dimensional structure that we want to recover, is often `perturbed' by observation uncertainty or statistical errors. Besides, there might be a sparse pattern corrupting the low dimensional structure, as in approximate factor models \citep{ChaRot82,StoWat02} and robust PCA \citep{DeBla03,CanLiMaWri11}.

A general way to study these problems is to consider
\begin{equation}\label{eqn::decmp1}
\widetilde{A} = A + S + N,
\end{equation}
where $A$ is a low rank matrix, $S$ is a sparse matrix, and $N$ is a random matrix regarded as random noise or estimation error, all of which have the same size $d_1 \times d_2$. Usually $A$ is regarded as the `signal' matrix we are primarily interested in, $S$ is some sparse contamination whose effect we want to separate from $A$, and $N$ is the noise (or estimation error in covariance matrix estimation).

The decomposition (\ref{eqn::decmp1}) forms the core of a flourishing literature on robust PCA \citep{ChaSanParWil11, CanLiMaWri11}, structured covariance estimation \citep{FanFanLv08,FanLiaMin13}, multivariate regression \citep{YuaEkiLuMon07} and so on. Among these works, a standard condition on $A$ is matrix incoherence \citep{CanLiMaWri11}. Let the singular value decomposition be
\begin{equation}\label{eqn:A}
A = U \Sigma V^T = \sum_{i=1}^r \sigma_i u_i v_i^T,
\end{equation}
where $r$ is the rank of $A$, the singular values are $\sigma_1 \ge \sigma_2 \ge \ldots \ge \sigma_r > 0$, and the matrices $U = [u_1,\ldots, u_r] \in \mathbb{R}^{d_1 \times r}$, $V = [v_1,\ldots, v_r] \in \mathbb{R}^{d_2 \times r}$ consist of the singular vectors. The coherences $\mu(U), \mu(V)$ are defined as
\begin{equation}\label{def::incoherence}
\mu(U) = \frac{d_1}{r} \max_i \sum_{j=1}^r U_{ij}^2, \qquad \mu(V) = \frac{d_2}{r} \max_i \sum_{j=1}^r V_{ij}^2,
\end{equation}
where $U_{ij}$ and $V_{ij}$ are the $(i, j)$ entry of $U$ and $V$, respectively.
It is usually expected that $\mu_0 := \max \{ \mu(U), \mu(V) \}$ is not too large, which means the singular vectors $u_i$ and $v_i$ are incoherent with the standard basis. This incoherence condition (\ref{def::incoherence}) is necessary for us to separate the sparse component $S$ from the low rank component $A$; otherwise $A$ and $S$ are not identifiable. Note that we do not need any incoherence condition on $UV^T$, which is different from \cite{CanLiMaWri11} and is arguably unnecessary \citep{Che15}.

Now we denote the eigengap $\gamma_0= \min\{ \sigma_i - \sigma_{i+1}: i = 1,\ldots,r\}$  where $\sigma_{r+1}:=0$ for notational convenience. Also we let $E = S + N$, and view it as a perturbation matrix to the matrix $A$ in \eqref{eqn::decmp1}. To quantify the perturbation, we define a rescaled measure as $\tau_0 := \max\{ \sqrt{d_2/d_1}\| E \|_{1}, \sqrt{d_1/d_2} \| E \|_{\infty} \}$, where
\begin{equation}\label{def:matone}
\| E \|_{1} = \max_j \sum_{i=1}^{d_1} |E_{ij}|, \quad \| E \|_{\infty} = \max_i \sum_{j=1}^{d_2} |E_{ij}|,
\end{equation}
which are commonly used norms gauging sparsity \citep{BicLev08}. They are also operator norms in suitable spaces (see Section \ref{sec::2}). The rescaled norms $\sqrt{d_2/d_1}\| E \|_{1}$ and $\sqrt{d_1/d_2}\| E \|_{\infty}$ are comparable to the spectral norm $\| E \|_2 := \max_{\| u \|_2 = 1} \| E u \|_2$ in many cases; for example, when $E$ is an all-one matrix, $\sqrt{d_2/d_1}\| E \|_{1} = \sqrt{d_1/d_2} \| E \|_{\infty} = \| E \|_2$.

Suppose the perturbed matrix $\widetilde{A}$ also has the singular value decomposition:
\begin{equation}\label{eqn:Atilde}
\widetilde{A} = \sum_{i=1}^{d_1 \wedge d_2} \widetilde{\sigma}_i \widetilde{u}_i \widetilde{v}_i^T,
\end{equation}
where $\widetilde{\sigma}_i$ are nonnegative and in the decreasing order, and the notation $\wedge$ means $a \wedge b = \min \{a,b\}$. Denote $\widetilde{U} = [\widetilde{u}_1,\ldots, \widetilde{u}_r],V = [\widetilde{v}_1, \ldots, \widetilde{v}_r]$, which are counterparts of top $r$ singular vectors of $A$.

We will present an $\ell_\infty$ matrix perturbation result that bounds $\| \widetilde{u}_i - u_i \|_{\infty}$ and $\| \widetilde{v}_i - v_i \|_{\infty}$ up to sign.\footnote{`Up to sign' means we can appropriately choose an eigenvector or singular vector $u$ to be either $u$ or $-u$ in the bounds. This is becuase eigenvectors and singular vectors are not unique.} This result is different from $\ell_2$ bounds, Frobenius-norm bounds, or the $\sin \Theta$ bounds, as the $\ell_\infty$ norm is not orthogonal invariant. The following theorem is a simplified version of our main results in Section \ref{sec::2}.

\begin{thm} \label{thm::main1}
Let $\widetilde{A} = A + E$ and suppose the singular decomposition in (\ref{eqn:A}) and (\ref{eqn:Atilde}). Denote $\gamma_0= \min\{ \sigma_i - \sigma_{i+1}: i = 1,\ldots,r\}$  where $\sigma_{r+1}:=0$. Then there exists $C(r,\mu_0) = O(r^4 \mu_0^2)$ such that, if $\gamma_0 > C(r,\mu_0) \tau_0$, up to sign,
\begin{equation}\label{ineqn::svd1}
\max_{1\le i \le r} \| \widetilde{u}_i - u_i \|_{\infty} \le C(r,\mu_0) \frac{\tau_0}{\gamma_0 \sqrt{d_1}}  \;\;\; \text{and} \;\;\;
\max_{1\le i \le r} \| \widetilde{v}_i - v_i \|_{\infty} \le C(r,\mu_0) \frac{\tau_0}{\gamma_0 \sqrt{d_2}},
\end{equation}
where $\mu_0 = \max \{ \mu(U), \mu(V) \}$ is the coherence given after \eqref{def::incoherence} and $\tau_0 := \max\{ \sqrt{d_2/d_1}\| E \|_{1}, \sqrt{d_1/d_2} \| E \|_{\infty} \}$.
\end{thm}

When $A$ is symmetric,  the condition on the eigengap is simply $\gamma_0 > C(r,\mu_0) \|E\|_{\infty}$. It naturally holds for a variety of applications, where the low rank structure emerges as a consequence of a few factors driving the data matrix. For example, in Fama-French factor models, the excess returns in a stock market are driven by a few common factors \citep{FamFre93}; in collaborative filtering, the ratings of users are mostly determined by a few common preferences \citep{RenSre05}; in video surveillance, $A$ is associated with the stationary background across image frames \citep{OliRosPen00}. We will have a detailed discussion in Section \ref{sec::which}.

The eigenvector perturbation was studied by \cite{DavKah70}, where Hermitian matrices were considered, and the results were extended by \cite{Wed72} to general rectangular matrices. To compare our result with these classical results, assuming $\gamma_0 \ge 2 \| E \|_2$, a combination of Wedin's theorem and Mirsky's inequality \citep{Mir60} (the counterpart of Weyl's inequality for singular values) implies

\begin{equation}\label{ineqn::davisKahan2}
\max_{1\le k \le r} \big\{  \| v_k -  \widetilde{v}_k \|_{2} \vee \| u_k - \widetilde{u}_k \|_{2} \big\} \le \frac{2\sqrt{2} \| E \|_2}{\gamma_0}.
\end{equation}
where $a \vee b := \max\{a, b\}$.

\cite{YuWanSam15} also proved a similar bound as in (\ref{ineqn::davisKahan2}), and that result is more convenient to use. If we are interested in the $\ell_\infty$ bound but naively use the trivial inequality $\| x \|_{\infty} \le \| x \|_{2}$, we would have a suboptimal bound $O(\| E \|_2/\gamma_0)$ in many situations, especially in cases where $\| E \|_2$ is comparable to $\| E\|_{\infty}$. Compared with (\ref{ineqn::svd1}), the bound is worse by a factor of $\sqrt{d_1}$ for $u_k$ and $\sqrt{d_2}$ for $v_k$. In other words, converting the $\ell_2$ bound from Davis-Kahan theorem directly to the $\ell_{\infty}$ bound does not give a sharp result in general, in the presence of incoherent and low rank structure of $A$.  Actually, assuming $\|E \|_2$ is comparable with $\| E \|_{\infty}$, for square matrices, our $\ell_{\infty}$ bound (\ref{ineqn::svd1}) matches the $\ell_2$ bound (\ref{ineqn::davisKahan2}) in terms of dimensions $d_1$ and $d_2$. This is because $\| x \|_2 \le \sqrt{n}\, \|x \|_{\infty}$ for any $x \in \mathbb{R}^n$, so we expect to gain a factor $\sqrt{d_1}$ or $\sqrt{d_2}$ in those $\ell_\infty$ bounds. The intuition is that, when $A$ has an incoherent and low-rank structure,  the perturbation of singular vectors is not concentrated on a few coordinates.

To understand how matrix incoherence helps, let us consider a simple example with no matrix incoherence, in which (\ref{ineqn::davisKahan2}) is tight up to a constant. Let $A = d(1,0,\ldots,0)^T(1,0,\ldots,0)$ be a $d$-dimensional square matrix, and $E = d(0, 1/2, 0, \ldots, 0)^T (1,0,\ldots,0)$ of the same size. It is apparent that $\gamma_0 = d, \tau_0 = d/2$, and that $v_1 = (1,0,\ldots,0)^T, \widetilde{v}_1 = (2/\sqrt{5}, 1/\sqrt{5}, 0, \ldots, 0)^T$ up to sign. Clearly, the perturbation $\| \widetilde{v}_1 - v_1 \|_{\infty}$ is not vanishing as $d$ tends to infinity in this example, and thus, there is no hope of a strong upper bound as in (\ref{ineqn::svd1}) without the incoherence condition.

The reason that the factor $\sqrt{d_1}$ or $\sqrt{d_2}$ comes into play in (\ref{ineqn::davisKahan2}) is that, the error $u_k -  \widetilde{u}_k$ (and similarly for $v_k$) spreads out evenly in $d_1$ (or $d_2$) coordinates, so that the $\ell_{\infty}$ error is far smaller than the $\ell_{2}$ error. This, of course, hinges on the incoherence condition, which in essence precludes eigenvectors from aligning with any coordinate.

Our result is very different from the sparse PCA literature, in which it is usually assumed that the leading eigenvectors are sparse. In \cite{JohLu09}, it is proved that there is a threshold for $p/n$ (the ratio between the dimension and the sample size), above which PCA performs poorly, in the sense that $\langle \widetilde{v}_1, v_1 \rangle$ is approximately $0$. This means that the principal component computed from the sample covariance matrix reveals nothing about the true eigenvector. In order to mitigate this issue, in \cite{JohLu09} and subsequent papers \citep{VuLei12,Ma13,BerRig13}, sparse leading eigenvectors are assumed. However, our result is different, in the sense that we require a stronger eigengap condition $\gamma_0 > C(r,\mu_0) \|E \|_{\infty}$ (i.e.\ stronger signal), whereas in \cite{JohLu09}, the eigengap of the leading eigenvectors is a constant times $\| E\|_2$. This explains why it is plausible to have a strong uniform eigenvector perturbation bound in this paper.

We will illustrate the power of this perturbation result using robust covariance estimation as one application. In the approximate factor model, the true covariance matrix admits a decomposition into a low rank part $A$ and a sparse part $S$. Such models have been widely applied in finance, economics, genomics, and health to explore correlation structure.

However, in many studies, especially financial and genomics applications, it is well known that the observations exhibit heavy tails \citep{GupVarBod13}. This problem can be resolved with the aid of recent results of concentration bounds in robust estimation \citep{Cat12, HsuSab14, FanLiWan16}, which produces the estimation error $N$ in (\ref{eqn::decmp1}) with an optimal entry-wise bound. It nicely fits our perturbation result, and we can tackle it easily by following the ideas in \cite{FanLiaMin13}.

Here are a few notations in this paper. For a generic $d_1$ by $d_2$ matrix, the matrix max-norm is denoted as $\| M \|_{\max} = \max_{i,j}|M_{ij}|$. The matrix operator norm induced by vector $\ell_{p}$ norm is $\| M \|_p = \sup_{\| x \|_p = 1} \| M x \|_p$ for $1\le p \le \infty$. In particular, $\| M \|_1 = \max_j \sum_{i=1}^{d_1} |M_{ij}|$; $\| M \|_{\infty} = \max_i \sum_{j=1}^{d_2} |M_{ij}|$; and $\| \cdot \|$ denotes the spectral norm, or the matrix $2$-norm $\| \cdot \|_2$ for simplicity. We use $\sigma_j(M)$ to denote the $j^{th}$ largest singular value. For a symmetric matrix $M$, denote $\lambda_j(M)$ as its $j^{th}$ largest eigenvalue. If $M$ is a positive definite matrix, then $M^{1/2}$ is the square root of $M$, and $M^{-1/2}$ is the square root of $M^{-1}$.

\section{The $\ell_\infty$ perturbation result}\label{sec::2}

\subsection{Symmetric matrices}\label{sec::sym}

First, we study $\ell_\infty$ perturbation for symmetric matrices (so $d_1 = d_2$). The approach we study symmetric matrices will be useful to analyze asymmetric matrices, because we can always augment a $d_1 \times d_2$ rectangular matrix into a $(d_1+d_2) \times (d_1+d_2)$ symmetric matrix, and transfer the study of singular vectors to the eigenvectors of the augmented matrix. This augmentation is called \textit{Hermitian dilation}. \citep{Tro12, Pau02} 

Suppose that $A \in \mathbb{R}^{d \times d}$ is an $d$-dimensional symmetric matrix. The perturbation matrix $E\in \mathbb{R}^{d \times d}$ is also $d$-dimensional and symmetric. Let the perturbed matrix be $\widetilde{A} := A + E$. Suppose the spectral decomposition of $A$ is given by
\begin{equation}\label{eq:specA}
A  =  [V, V_\bot] \left( \begin{array}{cc}
\Lambda_1 & 0 \\ 0 & \Lambda_2 \end{array} \right) [V,V_\bot]^T = \sum_{i=1}^r \lambda_i v_i v_i^T + \sum_{i>r}  \lambda_i v_i v_i^T,
\end{equation}
where $\Lambda_1 = \text{diag}\{ \lambda_1,\ldots,\lambda_r \}$, $\Lambda_2 = \diag\{ \lambda_{r+1},\ldots, \lambda_n \}$, and where $|\lambda_1| \ge  |\lambda_2| \ge \ldots \ge | \lambda_n |$. Note the best rank-$r$ approximation of $A$ under the Frobenius norm is $A_r := \sum_{i \le r} \lambda_i v_i v_i^T$.\footnote{This is a consequence of Wielandt-Hoffman theorem.}
Analogously, the spectral decomposition of $\widetilde{A}$ is
\begin{equation*}
\widetilde{A}  = \sum_{i=1}^r \widetilde{\lambda}_i \widetilde{v}_i \widetilde{v}_i^T + \sum_{i>r}  \widetilde{\lambda}_i \widetilde{v}_i \widetilde{v}_i^T,
\end{equation*}
and write $\widetilde{V} = [\widetilde{v}_1, \ldots, \widetilde{v}_r] \in \mathbb{R}^{d \times r}$, where $|\widetilde{\lambda}_1| \ge  |\widetilde{\lambda}_2| \ge \ldots \ge | \widetilde \lambda_n |$.
Recall that $\| E \|_{\infty}$ given by (\ref{def:matone}) is an operator norm in the $\ell_\infty$ space, in the sense that $\| E \|_{\infty} = \sup_{\| u \|_\infty \le 1} \| E u \|_{\infty}$. This norm is the natural counterpart of the spectral norm $\| E \|_2 := \sup_{\| u \|_2 \le 1} \| E u \|_{2}$.

We will use notations $O(\cdot)$ and $\Omega(\cdot)$ to hide absolute constants.\footnote{We write $a = O(b)$ if there is a constant $C>0$ such that $a < Cb$; and $a = \Omega(b)$ if there is a constant $C'>0$ such that $a > C' b$.} The next theorem bounds the perturbation of eigenspaces up to a rotation.

\begin{thm}\label{thm::symbulk}
Suppose $| \lambda_r |  - \varepsilon = \Omega(r^3 \mu^2 \| E \|_{\infty} )$, where  $\varepsilon = \| A -  A_r \|_{\infty}$, which is the approximation error measured under the matrix $\infty$-norm and $\mu = \mu(V)$ is the coherence of $V$ defined in (\ref{def::incoherence}). Then, there exists an orthogonal matrix $R \in \mathbb{R}^{r \times r}$ such that
\begin{equation*}
\| \widetilde{V} R - V \|_{\max} = O\left( \frac{r^{5/2}\mu^2 \| E \|_{\infty} }{(| \lambda_r |  - \varepsilon) \sqrt{d}} \right).
\end{equation*}
\end{thm}

This result involves an unspecified rotation $R$, due to the possible presence of multiplicity of eigenvalues. In the case where $\lambda_1=\cdots=\lambda_r > 0$, the individual eigenvectors of $V$ are only identifiable up to rotation. However, assuming an eigengap (similar to Davis-Kahan theorem), we are able to bound the perturbation of individual eigenvectors (up to sign).

\begin{thm}\label{thm::symindiv}
Assume the conditions in Theorem \ref{thm::symbulk}. In addition, suppose $\delta$ satisfies $\delta > \| E \|_2$,  and for any $i \in [r]$, the interval $[\lambda_i - \delta, \lambda_i + \delta] $ does not contain any eigenvalues of $A$ other than $\lambda_i$. Then, up to sign,
\begin{equation*}
\max_{i \in [r]} \| \widetilde{v}_i - v_i \|_\infty = \| \widetilde{V} - V \|_{\max} =  O\left( \frac{r^4 \mu^2 \| E \|_{\infty}}{(| \lambda_r |  - \varepsilon) \sqrt{d}} + \frac{r^{3/2}\mu^{1/2} \| E \|_2}{\delta \sqrt{d}} \right).
\end{equation*}
\end{thm}

To understand the above two theorems, let us consider the case where $A$ has exactly rank $r$ (i.e., $\varepsilon = 0$), and $r$ and $\mu$ are not large (say, bounded by a constant). Theorem \ref{thm::symbulk} gives a uniform entrywise bound $O(\| E \|_{\infty} /| \lambda_r |\sqrt{d})$ on the eigenvector perturbation. As a comparison, the Davis--Kahan $\sin \Theta$ theorem \citep{DavKah70} gives a bound $O(\| E \|_2 /| \lambda_r |)$ on $\| \widetilde{V} R - V \|_2$ with suitably chosen rotation $R$.\footnote{To see how the Davis-Kahan $\sin \Theta$ theorem relates to this form, we can use the identity $ \| \sin \Theta(\widetilde{V}, V)  \|_2 = \| \widetilde{V} \widetilde{V}^T - VV^T \|_2$ \citep{Ste90}, and the (easily verifiable) inequality $2 \min_R \| \widetilde{V}R - V \|_2 \ge  \| \widetilde{V} \widetilde{V}^T - VV^T \|_2 \ge  \min_R \| \widetilde{V}R - V \|_2$ where $R$ is an orthogonal matrix.}  This is an order of $\sqrt{d}$ larger than the bound given in Theorem \ref{thm::symbulk} when $\|E\|_\infty$ is of the same order as $\|E\|_2$.
Thus, 
in scenarios where $ \| E \|_2 $ is comparable to $\| E \|_{\infty}$, this is a refinement of Davis-Kahan theorem, because the max-norm bound in Theorem \ref{thm::symbulk} provides an entry-wise control of perturbation.  Although $ \| E \|_{\infty} \ge \| E \|_{2} $,\footnote{Since $\| E \|_1 \|E \|_{\infty} \le \| E \|_2^2$ \citep{Ste90}, the inequality follows from $\| E \|_1 = \| E \|_{\infty}$ by symmetry.} there are many settings where the two quantities are comparable; for example, if $E$ has a submatrix whose entries are identical and has zero entries otherwise, then $ \| E \|_{\infty} = \| E \|_{2}$.

Theorem \ref{thm::symindiv} provides the perturbation of individual eigenvectors, under a usual eigengap assumption. When $r$ and $\mu$ are not large, we incur an additional term $O(\| E \|_2 / \delta \sqrt{d})$ in the bound. This is understandable, since $\| \widetilde{v}_i - v_i \|_2$ is typically $O(\| E \|_2 / \delta)$.

When the rank of $A$ is not exactly $r$, we require that $| \lambda_r |$ is larger than the approximation error $\| A - A_r \|_{\infty}$. It is important to state that this assumption is more restricted than the eigengap assumption in the Davis-Kahan theorem, since $\| A - A_r \|_{\infty} \ge \| A - A_r \|_{2} = | \lambda_{r+1}|$. However, different from the matrix max-norm, the spectral norm $\| \cdot \|_2$ only depends on the eigenvalues of a matrix, so it is natural to expect $\ell_2$ perturbation bounds that only involve $\lambda_r$ and $\lambda_{r+1}$. It is not clear whether we should expect an $\ell_\infty$ bound that involves $\lambda_{r+1}$ instead of $\varepsilon$. More discussions can be found in Section \ref{sec::org}.

We do not pursue the optimal bound in terms of $r$ and $\mu(V)$ in this paper, as the two quantities are not large in many applications, and the current proof is already complicated.

\subsection{Rectangular matrices}\label{sec::asymm}

Now we establish $\ell_\infty$ perturbation bounds for general rectangular matrices. The results here are more general than those in Section \ref{sec::intro}, and in particular, we allow the matrix $A$ to be of approximate low rank. Suppose that both $A$ and $E$ are $d_1 \times d_2$ matrices, and $\widetilde{A} := A + E$. The rank of $A$ is at most $d_1 \wedge d_2$ (where $a \wedge b = \min \{a, b\}$). Suppose an integer $r$ satisfies $r \le \mathrm{rank}(A)$. Let the singular value decomposition of $A$ be

\begin{equation*}
A  =   \sum_{i=1}^r \sigma_i u_i v_i^T + \sum_{i=r+1}^{d_1 \wedge d_2}  \sigma_i u_i v_i^T,
\end{equation*}
where the singular values are ordered as $\sigma_1 \ge \sigma_2 \ge \ldots \ge \sigma_{d_1 \wedge d_2} \ge 0$, and the unit vectors $u_1, \ldots, u_{d_1 \wedge d_2}$ (or unit vectors $v_1, \ldots, v_{d_1 \wedge d_2}$) are orthogonal to each other. We denote $U = [u_1,\ldots,u_r] \in \mathbb{R}^{d_1 \times r}$ and $V = [v_1,\ldots,v_r] \in \mathbb{R}^{d_2 \times r}$. Analogously, the singular value decomposition of $\widetilde{A}$ is
\begin{equation*}
\widetilde{A}  =   \sum_{i=1}^r \widetilde{\sigma}_i \widetilde{u}_i  \widetilde{v}_i^T + \sum_{i=r+1}^{d_1 \wedge d_2}  \widetilde{\sigma}_i \widetilde{u}_i \widetilde{v}_i^T,
\end{equation*}
where $\widetilde{\sigma}_1 \ge \ldots \ge \widetilde{\sigma}_{d_1 \wedge d_2}$. Similarly, columns of $\widetilde{U} = [\widetilde{u}_1,\ldots,\widetilde{u}_r] \in \mathbb{R}^{d_1 \times r}$ and $V = [\widetilde{v}_1,\ldots,\widetilde{v}_r] \in \mathbb{R}^{d_2 \times r}$ are orthonormal.

Define $\mu_0 = \max \{ \mu(V), \mu(U) \}$, where $\mu(U)$ (resp.\ $\mu(V)$) is the coherence of $U$ (resp.\ $V$). This $\mu_0$ will appear in the statement of our results, as it controls both the structure of left and right singular spaces. When, specially, $A$ is a symmetric matrix, the spectral decomposition of $A$ is also the singular value decomposition (up to sign), and thus $\mu_0$ coincides with $\mu$ defined in Section \ref{sec::sym}.

Recall the definition of matrix $\infty$-norm and $1$-norm of a rectangular matrix (\ref{def:matone}).
Similar to the matrix $\infty$-norm, $\| \cdot \|_{1}$ is an operator norm in the $\ell_1$ space. An obvious relationship between matrix $\infty$-norm and $1$-norm is $\| E \|_{\infty} = \| E^T \|_1$. Note that the matrix $\infty$-norm and $1$-norm have different number of summands in their definitions, so we are motivated to consider $\tau_0 := \max\{ \sqrt{d_1/d_2} \|E\|_{\infty} , \sqrt{d_2/d_1} \| E \|_{1} \}$ to balance the dimensions $d_1$ and $d_2$.

Let $A_r = \sum_{i \le r} \sigma_i u_i v_i^T$ be the best rank-$r$ approximation of $A$ under the Frobenius norm, and let $\varepsilon_0 = \sqrt{d_1/d_2} \|A-A_r\|_{\infty} \vee \sqrt{d_2/d_1} \| A-A_r \|_{1}$, which also balances the two dimensions. Note that in the special case where $A$ is symmetric, this approximation error $\varepsilon_0$ is identical to $\varepsilon$ defined in Section \ref{sec::sym}. The next theorem bounds the perturbation of singular spaces.

\begin{thm}\label{thm::assymbulk}
Suppose that $\delta_0 - \varepsilon_0 = \Omega(r^3 \mu_0^2 \tau_0)$. Then, there exists orthogonal matrices $R_U, R_V \in \mathbb{R}^{r \times r}$ such that,
\begin{align*}
\| \widetilde{U} R_U - U \|_{\max} = O\Big( \frac{r^{5/2}\mu_0^{2} \tau_0}{(\sigma_r - \epsilon_0)\sqrt{d_1} }\Big), \qquad \| \widetilde{V} R_V - V \|_{\max} = O\Big( \frac{r^{5/2}\mu_0^{2} \tau_0}{(\sigma_r - \epsilon_0)\sqrt{d_2} }\Big).
\end{align*}
\end{thm}

Similar to Theorem \ref{thm::symindiv}, under an assumption of gaps between singular values, the next theorem bounds the perturbation of individual singular vectors.

\begin{thm}\label{thm::assymindiv}
Suppose the same assumption in Theorem \ref{thm::assymbulk}. In addition, suppose $\delta_0$ satisfies $\delta_0 > \| E \|_2$,  and for any $i \in [r]$, the interval $[\sigma_i - \delta_0, \sigma_i + \delta_0] $ does not contain any eigenvalues of $A$ other than $\sigma_i$. Then, up to sign,
\begin{align}\label{ineqn::assymindiv}
\max_{i \in [r]} \| \widetilde{u}_i - u_i \|_\infty &= O \Big( \frac{r^4\mu_0^2 \tau_0}{(\sigma_r - \varepsilon_0)\sqrt{d_1}} + \frac{r^{3/2}\mu_0^{1/2}\|E\|_2}{\delta_0\sqrt{d_1}}  \Big), \\
\max_{i \in [r]} \| \widetilde{v}_i - v_i \|_\infty &= O \Big( \frac{r^4\mu_0^2 \tau_0}{(\sigma_r - \varepsilon_0)\sqrt{d_2}} + \frac{r^{3/2}\mu_0^{1/2}\|E\|_2}{\delta_0\sqrt{d_2}}  \Big).
\end{align}
\end{thm}

As mentioned in the beginning of this section, we will use dilation to augment all $d_1 \times d_2$ matrices into symmetric ones with size $d_1 + d_2$. In order to balance the possibly different scales of $d_1$ and $d_2$, we consider a weighted max-norm. This idea will be further illustrated in Section \ref{sec::org}.

\subsection{Examples: which matrices have such structure?}\label{sec::which}

In many problems, low-rank structure naturally arises due to the impact of pervasive latent factors that influence most observed data. Since observations are imperfect, the low-rank structure is often `perturbed' by an additional sparse structure, gross errors, measurement noises, or the idiosyncratic components that can not be captured by the latent factors. We give some motivating examples with such structure.

\medskip
\textit{Panel data in stock markets.} Consider the excess returns from a stock market over a period of time. The driving factors in the market are reflected in the covariance matrix as a low rank component $A$. The residual covariance of the idiosyncratic components is often modeled by a sparse component $S$. Statistical analysis including PCA is usually conducted based on the estimated covariance matrix $\widetilde{A} = \widehat\Sigma$, which is perturbed  from the true covariance $\Sigma = A + S$ by the estimation error $N$ \citep{StoWat02,FanLiaMin13}. In Section \ref{sec::covEstFacMol}, we will develop a robust estimation method in the presence of heavy-tailed return data.

\medskip
\textit{Video surveillance.} In image processing and computer vision, it is often desired to separate moving objects from static background before further modeling and analysis \citep{OliRosPen00,HuTanWanMay04}. The static background corresponds to the low rank component $A$ in the data matrix, which is a collection of
video frames, each consisting of many pixels represented as a long vector in the data matrix. Moving objects and noise correspond to the sparse matrix $S$ and noise matrix $N$. Since the background is global information and reflected by many pixels of a frame, it is natural for the incoherence condition to hold.

\medskip
\textit{Wireless sensor network localization.} In wireless sensor networks, we are usually interested in determining the location of sensor nodes with unknown position based on a few (noisy) measurements between neighboring nodes \citep{DohPisSJGha01, BisYe04}. Let $\mathbb X$ be an $r$ by $n$ matrix such that each column $x_i$ gives the coordinates of each node in a plane ($r=2$) or a space ($r=3$). Assume the center of the sensors has been relocated at origin.
Then the low rank matrix $A =\mathbb X^T \mathbb X$,  encoding the true distance information, has to satisfy distance constraints given by the measurements.
The noisy distance matrix $\widetilde A$ after centering, equals to the sum of $A$ and a matrix $N$ consisting of measurement errors.
Suppose that each node is a random point uniformly distributed in a rectangular region. It is not difficult to see that with high probability, the top $r$ eigenvalues of $\mathbb X^T \mathbb X$ and their eigengap scales with the number of sensors $n$ and the leading eigenvectors have a bounded coherence.

In our theorems, we require that the coherence $\mu$ is not too large. This is a natural structural condition associated with the low rank matrices. Consider the following very simple example: if the eigenvectors $v_1,\ldots, v_r$ of the low rank matrix $A$ are uniform unit vectors in a sphere, then with high probability, $\max_i \| v_i \|_{\infty} = O(\sqrt{\log n})$, which implies $\mu = O(\log n)$. An intuitive way to understand the incoherence structure is that no coordinates of $v_1$ (or $v_2, \ldots v_r$) are dominant. In other words, the eigenvectors are not concentrated on a few coordinates.

In all our examples, the incoherence structure is natural. The factor model satisfies such structure, which will be discussed in Section \ref{sec:app}. In the video surveillance example, ideally, when the images are static, $A$ is a rank one matrix $x \mathbf{1}^T$. Since usually a majority of pixels (coordinates of $x$) help to display an image, the vector $x$ often has dense coordinates with comparable magnitude, so $A$ also has an incoherence structure in this example. Similarly, in the sensor localization example, the coordinates of all sensor nodes are comparable in magnitude, so the low rank matrix $A$ formed by $\mathbb X^T \mathbb X$ also has the desired incoherence structure.

\subsection{Other perturbation results}\label{sec::connect}

Although the eigenvector perturbation theory is well studied in numerical analysis, there is a renewed interest among statistics and machine learning communities recently, due to the wide applicability of PCA and other eigenvector-based methods. In \cite{CaiZha16, YuWanSam15}, they obtained variants or improvements of Davis-Kahan theorem (or Wedin's theorem), which are user-friendly in the statistical contexts. These results assume the perturbation is deterministic, which is the same as Davis-Kahan theorem and Wedin's theorem. In general, these results are sharp, even when the perturbation is random, as evidenced by the BBP transition \citep{BaiAroPec05}.

However, these classical results can be suboptimal, when the perturbation is random and the smallest eigenvalue gap $\lambda_1 - \lambda_2$ does not capture particular spectrum structure. For example, \cite{Vu11, ORoVuKe13} showed that with high probability, there are bounds sharper than the Wedin's theorem, when the signal matrix is low-rank and satisfies certain eigenvalue conditions.

In this paper, our perturbation results are deterministic, thus the bound can be suboptimal when the perturbation is random with certain structure (e.g. the difference between sample covariance and population one for i.i.d. samples). However, the advantage of a deterministic result is that it is applicable to any random perturbation.  This is especially useful when we cannot make strong random assumptions on the perturbation (e.g., the perturbation is an unknown sparse matrix). In Section \ref{sec:app}, we will see examples of this type.

\section{Application to robust covariance estimation} \label{sec:app}

We will study the problem of robust estimation of covariance matrices and show the strength of our perturbation result. Throughout this section, we assume both rank $r$ and the coherence $\mu(V)$ are bounded by a constant, though this assumption can be relaxed. We will use $C$ to represent a generic constant, and its value may change from line to line.

\subsection{PCA in spiked covariance model} \label{sec::covEstFacMol}

To initiate our discussions, we first consider sub-Gaussian random variables. Let $X = (X_1, \dots, X_d)$ be
a random $d$-dimensional vector with mean zero and covariance matrix
\begin{equation}\label{eqn:covStruc}
\Sigma = \sum_{i=1}^r \lambda_i {v}_i {v}_i^T + \sigma^2 I_d := \Sigma_1 + \Sigma_2, \quad \quad (\lambda_1 \ge \ldots \ge \lambda_r > 0),
\end{equation}
and $\mathbb X$ be an $n$ by $d$ matrix, whose rows are independently sampled from the same distribution.
This is the spiked covariance model that has received intensive study in recent years.
Let the empirical covariance matrix be $\widehat{\Sigma} = \mathbb X^T \mathbb X / n$. Viewing the empirical covariance matrix as its population version plus an estimation error, we have the decomposition
\begin{equation*}
\widehat{\Sigma} = \Sigma_1 + \Sigma_2 +  \Big(\frac{1}{n} \mathbb X^T \mathbb X - \Sigma \Big),
\end{equation*}
which is a special case of the general decomposition in \eqref{eqn::decmp1}. Here, $\Sigma_2$ is the sparse component, and the estimation error $\mathbb X^T \mathbb X/n - \Sigma$ is the noise component. Note that $v_1, \dots, v_r$ are just the top $r$ leading eigenvectors of $\Sigma$ and we write $V = [v_1, \dots, v_r]$. Assume the top $r$ eigenvectors of $\widehat\Sigma$ are denoted by $\widehat v_1,\dots, \widehat v_r$.  We want to find an $\ell_\infty$ bound on the estimation error $\widehat{v_i} - v_i$ for all $i \in [r]$.

When the dimension $d$ is comparable to or larger than $n$, it has been shown by \cite{JohLu09} that the leading empirical eigenvector $\widehat v_1$ is not a consistent estimate of the true eigenvector $v_1$, unless we assume larger eigenvalues. Indeed, we will impose more stringent conditions on $\lambda_i$'s in order to obtain good $\ell_{\infty}$ bounds.

Assuming the coherence $\mu(V)$ is bounded, we can easily see $\text{Var}(X_j) \le \sigma^2 + C \lambda_1/d$ for some constant $C$. It follows from the standard concentration result (e.g., \cite{Ver10}) that if rows of $\mathbb X$ contains  i.i.d sub-Gaussian vectors and $\log d = O(n)$, then with probability greater than $1 - d^{-1}$,
\begin{equation}\label{ineqn:empCov}
\| \frac{1}{n} \mathbb X^T \mathbb X - \Sigma \|_{\max} \le C \big(\sigma^2 + \frac{\lambda_1}{d}\big) \sqrt{\frac{\log d}{n}}.
\end{equation}
To apply Theorem \ref{thm::symindiv}, we treat $\Sigma_1$ as $A$ and $\widehat \Sigma - \Sigma_1$ as $E$. If the conditions in Theorem \ref{thm::symindiv} are satisfied, we will obtain
\begin{equation}\label{eq3.3}
    \max_{1\le k \le r} \|\widehat{v}_k - v_k\|_{\infty} = O (\|E\|_{\infty} /(\lambda_r \sqrt{d})  + \|E\|_2 / (\delta \sqrt{d}) ).
\end{equation}
Note there are simple bounds on $\|E\|_{\infty}$ and $\|E\|_2$:
\begin{equation*}
\|E\|_2 \le \|E\|_{\infty} \le \sigma^2 + d \, \| \frac{1}{n} \mathbb X^T \mathbb X- \Sigma \|_{\max} \le C \Big\{ 1 + \big(d \sigma^2 + \lambda_1 \big)\sqrt{\frac{\log d}{n}} \Big\}.
\end{equation*}
By assuming a strong uniform eigengap, the conditions in Theorem \ref{thm::symindiv} are satisfied, and the bound in (\ref{eq3.3}) can be simplified. Define the uniform eigengap as
$$
\gamma = \min \{  \lambda_i - \lambda_{i+1} : 1 \le i \le r \}, \qquad \lambda_{r+1} := 0.
$$
Note that $\gamma \le \min\{ \lambda_r, \delta \}$, so if $\gamma > C ( 1 + \big(d \sigma^2 + \lambda_1 \big)\sqrt{\log d/n} )$, we have
$$
 \max_{1\le k \le r} \|\widehat{v}_k - v_k\|_{\infty} = O_P \Big( \frac{\|E\|_{\infty}}{\gamma \sqrt{d}}\Big ) = O_P\Big( \frac{1 + \big(d \sigma^2 + \lambda_1 \big)\sqrt{\log d/n} }{\gamma \sqrt{d}} \Big),
$$
In particular, when $\lambda_1 \asymp \gamma$ and $\gamma \gg \max\{1, \sigma^2d\sqrt{\log d/n}\}$, we have
\begin{equation*}
\max_{1\le k \le r} \| \widehat{v}_k  - v_k\|_{\infty} = o_P\Big(\frac{1}{\sqrt{d}}\Big).
\end{equation*}
The above analysis pertains to the structure of sample covariance matrix.
In the following subsections, we will estimate the covariance matrix using more complicated robust procedure. Our perturbation theorems in Section \ref{sec::2} provide a fast and clean approach to obtain new results.

\medskip

\subsection{PCA for robust covariance estimation}\label{sec:app2}
The usefulness of Theorem \ref{thm::symindiv} is more pronounced when the random variables are heavy-tailed. Consider again the covariance matrix $\Sigma$ with structure (\ref{eqn:covStruc}). Instead of assuming sub-Gaussian distribution, we assume there exists a constant $C>0$ such that $\max_{j \le d} EX_j^4 < C$, i.e.\ the fourth moments of the random variables are uniformly bounded.

Unlike sub-Gaussian variables, there is no concentration bound similar to (\ref{ineqn:empCov}) for the empirical covariance matrix. Fortunately, thanks to recent advances in robust statistics (e.g., \cite{Cat12}), robust estimate of $\Sigma$ with guaranteed concentration property becomes possible. We shall use the method proposed in \cite{FanLiWan16}. Motivated by the classical $M$-estimator of \cite{Hub64}, \cite{FanLiWan16} proposed a robust estimator for each element of $\widehat{\Sigma}$, by solving a Huber loss based minimization problem
\begin{equation} \label{eqn:huber}
\widehat{\Sigma}_{ij} = \argmin_{\mu}  \sum_{t=1}^n  l_{\alpha}(X_{ti}X_{tj} - \mu),
\end{equation}
where $l_{\alpha}$ is the Huber loss defined as
\begin{equation*}
l_{\alpha} (x) = \begin{cases}
2 \alpha|x| - \alpha^{2}, & |x| \ge \alpha, \\
x^2, & |x| \le \alpha.
\end{cases}
\end{equation*}
The parameter $\alpha$ is suggested to be $\alpha = \sqrt{ nv^2 / \log(\epsilon^{-1})}$ for $\epsilon \in (0,1)$, where $v$ is assumed to satisfy $v \ge \max_{ij}\sqrt{\text{Var}(X_iX_j)}$. If $\log(\epsilon^{-1}) \le n/8$, \cite{FanLiWan16} showed
\begin{equation*}
P\Big(| \widehat{\Sigma}_{ij} - \Sigma_{ij}| \le 4v \sqrt{\frac{\log(\epsilon^{-1})}{n}} \Big) \ge 1 - 2\epsilon.
\end{equation*}

From this result, the next proposition is immediate by taking $\epsilon = d^{-3}$.

\begin{prop} \label{prop3.1}
Suppose that there is a constant $C$ with $\max_{j\le d} EX_j^4 < C$. Then with probability greater than $1-d^{-1}(1+d^{-1})$, the robust estimate of covariance matrix with $\alpha = \sqrt{ 3 nv^2  \log(d)}$ satisfies
\begin{equation*}
\| \widehat{\Sigma} - \Sigma \|_{\max} \le 4 v \sqrt{\frac{3 \log d}{n}},
\end{equation*}
where $v$ is a pre-determined parameter assumed to be no less than $\max_{ij}\sqrt{\text{Var}(X_iX_j)}$.
\end{prop}

This result relaxes the sub-Gaussianity assumption by robustifying the covariance estimate. It is apparent that the $\ell_\infty$ bound in the previous section is still valid in this case. To be more specific, suppose $\mu(V)$ is bounded by a constant.  Then, \eqref{eq3.3} holds for the PCA based on the robust covariance estimation. When $\lambda_1 \asymp \gamma$ and $\gamma \gg \max\{1, \sigma^2 d \sqrt{\log d /n}\}$, we again have
\begin{equation*}
\max_{1 \le k \le r} \| \widehat{v}_k - v_k\|_{\infty} =  O_P\Big( \frac{1 + \big(d \sigma^2 + \lambda_1 \big)\sqrt{\log d/n} }{\gamma \sqrt{d}} \Big) = o_P\Big(\frac{1}{\sqrt{d}}\Big).
\end{equation*}

Note that an entrywise estimation error $o_p(1 / \sqrt{d})$ necessarily implies consistency of the estimated eigenvectors, since we can easily convert an $\ell_\infty$ result into an $\ell_2$ result. The minimum signal strength (or magnitude of leading eigenvalues) for such consistency is shown to be $\sigma^2d/n$ under the sub-Gaussian assumption \citep{WanFan17}.

If the goal is simply to prove consistency of $\widehat{v}_k$, the strategy of using our $\ell_\infty$ perturbation bounds is not optimal. However, there are also merits: our result is nonasymptotic; it holds for more general distributions (beyond sub-Gaussian distributions); and its entrywise bound gives stronger guarantee. Moreover, the $\ell_\infty$ perturbation bounds provide greater flexibility for analysis, since it is straightforward to adapt analysis to problems with more complicated structure. For example, the above discussion can be easily extended to a general $\Sigma_2$ with bounded $\|\Sigma_2\|_{\infty}$ rather than a diagonal matrix.


\subsection{Robust covariance estimation via factor models} \label{sec::covEstFacMod}

In this subsection, we will apply Theorem \ref{thm::symindiv} to robust large covariance matrix estimation for approximate factor models in econometrics. With this theorem, 
we are able to extend the data distribution in factor analysis beyond exponentially decayed distributions considered by  \cite{FanLiaMin13}, to include heavy-tailed distributions.

Suppose the observation $y_{it}$, say, the excess return at day $t$ for stock $i$, admits a decomposition
\begin{equation}\label{eqn:decomp1}
y_{it} = b_i^T f_t + u_{it}, \quad \quad i\le d, t \le n,
\end{equation}
where $b_i \in \mathbb{R}^r$ is the unknown but fixed loading vector, $f_t \in \mathbb{R}^r$ denotes the unobserved factor vector at time $t$, and $u_{it}$'s represent the idiosyncratic noises. Let $y_t = (y_{1t}, \dots, y_{dt})^T$ and $u_t = (u_{1t},\dots, u_{dt})^T$ so that $y_t = B f_t + u_t$, where $B = (b_1, \dots, b_d)^T \in \mathbb{R}^{d \times r}$. Suppose that $f_t$ and $u_t$ are uncorrelated and centered random vectors, with bounded fourth moments, i.e., the fourth moments of all entries of $f_t$ and $u_t$ are bounded by some constant.  We assume $\{f_t, u_t\}$ are independent for $t$, although it is possible to allow for weak temporal dependence as in \cite{FanLiaMin13}.
From (\ref{eqn:decomp1}), we can decompose $\Sigma = \text{Cov}(y_t)$ into a low rank component and a residual component:
\begin{equation}\label{eqn:decomp2}
\Sigma = BB^T + \Sigma_u,
\end{equation}
where $\Sigma_u := \text{Cov}(u_t)$. To circumvent the identifiability issue common in latent variable models, here we also assume, without loss of generality, $\text{Cov}(f_t) = I_r$ and that $B^TB$ is a diagonal matrix, since rotating $B$ 
will not affect the above decomposition (\ref{eqn:decomp2}).

We will need two major assumptions for our analysis: (1) the factors are \textit{pervasive} in the sense of Definition \ref{def:per}, and (2) there is a constant $C>0$ such that $\|\Sigma_u^{-1}\|_2, \|\Sigma_u\|_2 \le C$, which are standard assumptions in the factor model literature. The pervasive assumption is reasonable in financial applications,  since the factors have impacts on a large fraction of the outcomes \citep{ChaRot82, Bai03}. If the factor loadings $\{b_i\}_{i=1}^d$ are regarded as random realizations from a bounded random vector, the assumption holds \citep{FanLiaMin13}.
\begin{defn}\label{def:per}
In the factor model (\ref{eqn:decomp1}), the factors are called pervasive if there is a constant $C>0$ such that $\|B\|_{\max} \le C$ and the eigenvalues of the $r$ by $r$ matrix $B^TB/d$ are distinct and bounded away from zero and infinity.
\end{defn}

Let $\{ \lambda_i, v_i \}_{i=1}^r$ be the top $r$ eigenvalues and eigenvectors of $\Sigma$, and similarly, $\{ \overline{\lambda}_i, \overline{v}_i \}_{i=1}^r$ for $BB^T$.  In the following proposition, we show that pervasiveness is naturally connected to the incoherence structure. This connects well between the econometrics and machine learning literatures and provide a good interpretation on the concept of the incoherence.  Its proof can be found in the appendix.

\begin{prop}\label{prop:fm}
Suppose there exists a constant $C>0$ such that $\|\Sigma_u\| \le C$. The factors $f_t$ are pervasive if and only if the coherence $\mu(V)$ for $V = (v_1, \dots, v_r) \in \mathbb{R}^{d \times r}$ is bounded by some constant, and $\lambda_i = \lambda_i(\Sigma) \asymp d$ for $i \le r$ so that $\min_{1 \le i \ne j \le r} |\lambda_i - \lambda_j| / \lambda_{j} > 0$.
\end{prop}

Our goal is to obtain a good covariance matrix estimator by exploiting the structure (\ref{eqn:decomp2}). Our strategy is to use a generalization of the principal orthogonal complement thresholding (POET) method proposed in \cite{FanLiaMin13}. The generic POET procedure encompasses three steps:
\begin{itemize}
\item[(1)] Given three pilot estimators $\widehat\Sigma, \widehat\Lambda = \diag(\widehat\lambda_1, \dots, \widehat\lambda_r), \widehat V = (\widehat v_1, \dots, \widehat v_r)$  respectively for true covariance $\Sigma$, leading eigenvalues $\Lambda = \diag(\lambda_1, \dots, \lambda_r)$ and leading eigenvectors $V = (v_1, \dots, v_r)$, compute the principal orthogonal complement $\widehat\Sigma_u$:
\begin{equation}
    \widehat\Sigma_u = \widehat\Sigma -  \widehat V \widehat\Lambda \widehat V^T \,.
\end{equation}

\item[(2)] Apply the correlation thresholding to $\widehat\Sigma_u$ to obtain thresholded estimate $\widehat\Sigma_u^{\top}$ defined as follows:
\begin{equation} \label{thresholding}
\widehat\Sigma_{u,ij}^{\top} = \left\{  \begin{array}{lr} \widehat\Sigma_{u,ij}, & i = j\\
s_{ij} (\widehat\Sigma_{u,ij}) I(|\widehat\Sigma_{u,ij}| \ge \tau_{ij}),  & i \ne j\end{array} \right.,
\end{equation}
where $s_{ij}(\cdot)$ is the generalized shrinkage function \citep{AntFan01,RotLevZhu09} and $\tau_{ij} = \tau (\hat\sigma_{u,ii} \hat\sigma_{u,jj})^{1/2}$ is an entry-dependent threshold. $\tau$ will be determined later in Theorem \ref{suff}.  This step exploits the sparsity of $\Sigma_u$.

\item[(3)] Construct the final estimator $\widehat\Sigma^{\top} = \widehat V \widehat\Lambda \widehat V^T + \widehat\Sigma_u^{\top}$.
\end{itemize}

The key feature in the above procedure lies in the flexibility of choosing the pilot estimators in the first step. We will choose $\widehat \Sigma$ according to data generating distribution. Typically we can use $\hat \lambda_i, \hat v_i$ for $i \le r$ as the eigenvalues/vectors of $\widehat\Sigma$. However, $\hat{\Lambda}$ and $\hat{V}$ in general do not have to come from the spectral information of $\widehat\Sigma$ and can be obtained separately via different methods. 

To guide the selection of proper pilot estimators, \cite{FanLiuWan17} provided a high level sufficient condition for this simple procedure to be effective, and its performance is gauged, in part, through the sparsity level of $\Sigma_u$, defined as $m_d := \max_{i \le d} \sum_{j \le d} |\Sigma_{u,ij}|^q$. When $q = 0$, $m_d$ corresponds to the maximum number of nonzero elements in each row of $\Sigma_u$. For completeness, we present the theorem given by \cite{FanLiuWan17} in the following.

\begin{thm} \label{suff}
Let $w_n = \sqrt{\log d/n} + 1/\sqrt{d}$.
Suppose there exists $C > 0$ such that $\|\Sigma_u^{-1}\|, \|\Sigma_u \| \le C$ and we have pilot estimators $\widehat\Sigma,  \widehat\Lambda, \widehat V$ satisfying
\begin{align}
& \| \widehat{\Sigma} - \Sigma \|_{\max}  = O( \sqrt{\log d /n }) , \label{suff1}\\
& | \widehat{\lambda}_i / \lambda_i - 1 |  = O( \sqrt{\log d /n }) , \label{suff2} \\
& \| \widehat{v}_i - v_i \|_{\infty}  = O( w_n/\sqrt{d}) . \label{suff3}
\end{align}
Under the pervasiveness condition of the factor model (\ref{eqn:decomp1}), with $\tau \asymp w_n$, if $m_d w_n^{1-q} = o(1)$, the following rates of convergence hold with the generic POET procedure:

\beq \label{rate1}
\|\widehat\Sigma_u^{\top} - \Sigma_u\|_{2}  = O\Big( m_d w_n^{1-q} \Big) = \|(\widehat\Sigma_u^{\top})^{-1} - {\Sigma_u}^{-1}\|_{2}\,,
\eeq
and
\beq \label{rate2}
\begin{aligned}
&\|\widehat\Sigma^{\top} - \Sigma\|_{\max} = O\Big( w_n \Big)\,, \\
&\|\widehat\Sigma^{\top} - \Sigma\|_{\Sigma} =  O\Big( \frac{\sqrt{d} \log d}{n} + m_d w_n^{1-q} \Big)\,, \\
&\|(\widehat\Sigma^{\top})^{-1} - \Sigma^{-1}\|_2 = O\Big( m_d w_n^{1-q} \Big)\,,
\end{aligned}
\eeq
where $\|A\|_{\Sigma} = d^{-1/2} \|\Sigma^{-1/2} A \Sigma^{-1/2}\|_F$ is the relative Frobenius norm.
\end{thm}

We remark that the additional term $1/\sqrt{d}$ in $w_n$, is due to the estimation of unobservable factors and is negligible when the dimensional $d$ is high. The optimality of the above rates of convergence is discussed in details in \cite{FanLiuWan17}.
Theorem \ref{suff} reveals a profound deterministic connection between the estimation error bound of the pilot estimators with the rate of convergences of the POET output estimators. Notice that the eigenvector estimation error is under the $\ell_{\infty}$ norm, for which our $\ell_{\infty}$ perturbation bounds will prove to be useful.

\medskip
\medskip
In this subsection, since we assume only bounded fourth moments, we choose $\widehat{\Sigma}$ to be the robust estimate of covariance matrix $\Sigma$ defined in (\ref{eqn:huber}).
We now invoke our $\ell_\infty$ bounds to show that the spectrum properties (eigenvalues and eigenvectors) are stable to perturbation.

Let us decompose $\widehat{\Sigma}$ into a form such that Theorem \ref{thm::symindiv} can be invoked:
\begin{equation*}
\widehat{\Sigma} =  \sum_{i=1}^r \overline{\lambda}_i \overline{v}_i \overline{v}_i^T + \Sigma_u + (\widehat{\Sigma} - \Sigma),
\end{equation*}
where $\widehat{\Sigma}$ is viewed as $\widetilde{A}$, the low-rank part $\sum_{i=1}^r \overline{\lambda}_i \overline{v}_i \overline{v}_i^T$, which is also $BB^T$, is viewed as $A$, and the remaining terms are treated as $E$. The following results follow immediately.

\begin{prop} \label{prop3.3}
Assume that there is a constant $C>0$ such that $ \| \Sigma_u \| \le C$. If the factors are pervasive, then with probability greater than $1 - d^{-1}$, we have
(\ref{suff1}) -- (\ref{suff3}) hold with $\widehat\lambda_i, \widehat v_i$ as the leading eigenvalues/vectors of $\widehat \Sigma$ for $i \le r$. In addition, (\ref{rate1}) and (\ref{rate2}) hold.
\end{prop}

The inequality (\ref{suff1}) follows directly from Proposition \ref{prop3.1} under the assumption of bounded fourth moments. It is also easily verifiable that (\ref{suff2}), (\ref{suff3}) follow from (\ref{suff1}) by Weyl's inequality and Theorem \ref{thm::symindiv} (noting that $\|\Sigma_u\|_{\infty} \le \sqrt{d} \|\Sigma_u\|$). See Section \ref{sec:app2} for more details.

Note that in the case of sub-Gaussian variables, sample covariance matrix and its leading eigenvalues/vectors will also serve the same purpose due to (\ref{ineqn:empCov}) and Theorem \ref{thm::symindiv} as discussed in Section \ref{sec::covEstFacMol}.

We have seen that the $\ell_{\infty}$ perturbation bounds are useful in robust covariance estimation, and particularly, they resolve a theoretical difficulty in the generic POET procedure for factor model based covariance matrix estimation.

\section{Simulations}

\subsection{Simulation: the perturbation result}
In this subsection, we implement numerical simulations to verify the perturbation bound in Theorem \ref{thm::symindiv}. We will show that the error behaves in the same way as indicated by our theoretical bound.

In the experiments, we let the matrix size $d$ run from $200$ to $2000$ by an increment of $200$. We fix the rank of $A$ to be $3$ ($r = 3$).
To generate an incoherence low rank matrix, we sample a $d \times d$ random matrix with iid standard normal variables, perform singular value decomposition, and extract the first $r$ right singular vectors $v_1, v_2, \ldots, v_r$. Let $V = (v_1,\ldots,v_r)$ and $D = \text{diag}(r\gamma, (r-1)\gamma, \ldots, \gamma)$ where $\gamma$ as before represents the eigengap. Then, we set $A = V D V^T$. By orthogonal invariance, $v_i$ is uniformly distributed on the unit sphere $\mathbb{S}^{d-1}$. It is not hard to see that with probability $ 1- O(d^{-1})$, the coherence of $V$ $\mu(V) = O(\sqrt{\log d})$.

We consider two types of sparse perturbation matrices $E$: (a) construct a $d \times d $ matrix $E_0$ by randomly selecting $s$ entries for each row, and sampling a uniform number in $[0,L]$ for each entry, and then symmetrize the perturbation matrix by setting $E = (E_0 + E_0^T)/2$; (b) pick $\rho \in (0,1), L' >0$, and let $E_{ij} = L'\rho^{|i - j|}$.
Note that in (b) we have $\| E \|_{\infty} \le 2L' / (1 - \rho)$, and thus we can choose suitable $L'$ and $\rho$ to control the $\ell_{\infty}$ norm of $E$. This covariance structure is common in cases where correlations between random variables depend on their ``distance'' $|i - j|$, which usually arises from autoregressive models.

\begin{figure}
    \centering
    \includegraphics[width=1\textwidth]{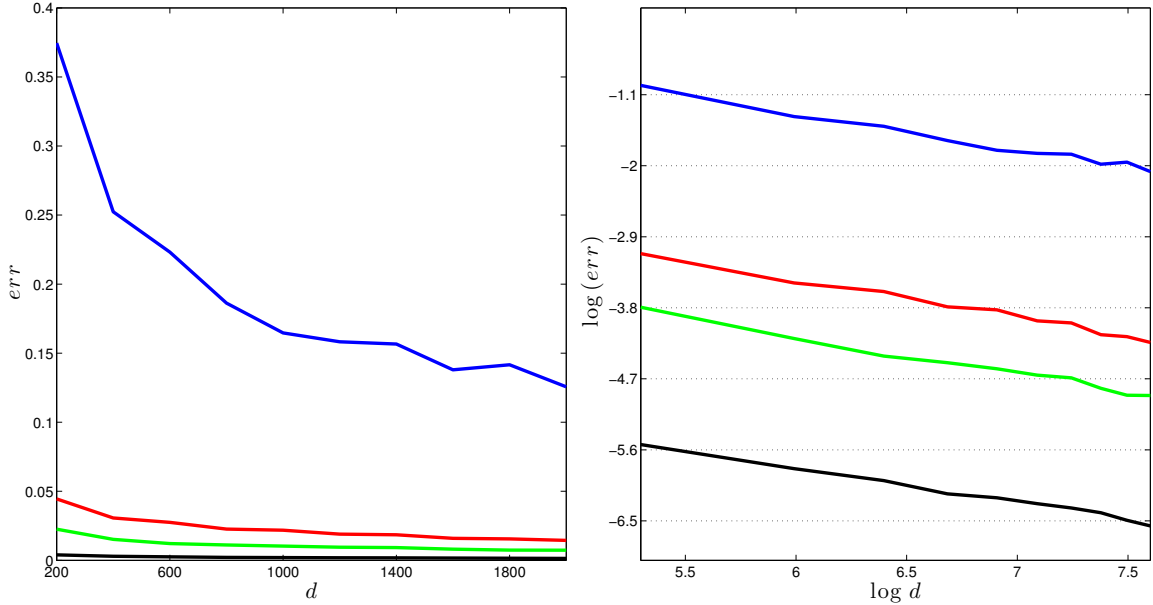}
    \caption{\textit{The left plot shows the perturbation error of eigenvectors against matrix size $d$ ranging from $200$ to $2000$, with different eigengap $\gamma$. The right plot shows $\log(err)$ against $\log(d)$. The slope is around $-0.5$. Blue lines represent $\gamma = 10$; red lines $\gamma = 50$; green lines $\gamma = 100$; and black lines $\gamma = 500$. We report the largest error over $100$ runs.}}
    \label{ptbFig1}
\end{figure}

 The perturbation of eigenvectors is measured by the element-wise error:
\begin{equation*}
err := \max_{1 \le i \le r} \min_{\eta_i \in \{ \pm 1 \}} \| \eta_i\widetilde{v}_i - v_i \|_{\infty},
\end{equation*}
where $\{ \widetilde{v}_i \}_{i = 1}^r$ are the eigenvectors of $\widetilde{A} = A + E$ in the descending order.

To investigate how the error depends on $\gamma$ and $d$, we generate $E$ according to mechanism (a) with $s = 10, L = 3$, and run simulations in different parameter configurations:  (1) let the matrix size $d$ range from $200$ to $2000$, and choose the eigengap $\gamma$ in $\{10, 50, 100, 500\}$ (Figure \ref{ptbFig1});
(2) fix the product $\gamma \sqrt{d}$ to be one of $\{2000, 3000, 4000, 5000\}$, and let the matrix size $d$ run from $200$ to $2000$ (Figure \ref{ptbFig2}).

To find how the errors behave for $E$ generated from different methods, we run simulations as in (1) but generate $E$ differently. We construct $E$ through mechanism (a) with $L = 10, s=3$ and $L = 0.6, s = 50$, and also through mechanism (b) with $L' = 1.5, \rho = 0.9$ and $L' = 7.5, \rho = 0.5$ (Figure \ref{ptbFig3}). The parameters are chosen such that $\|E\|_{\infty}$ is about $30$.

\begin{figure}
    \centering
    \includegraphics[width=1\textwidth]{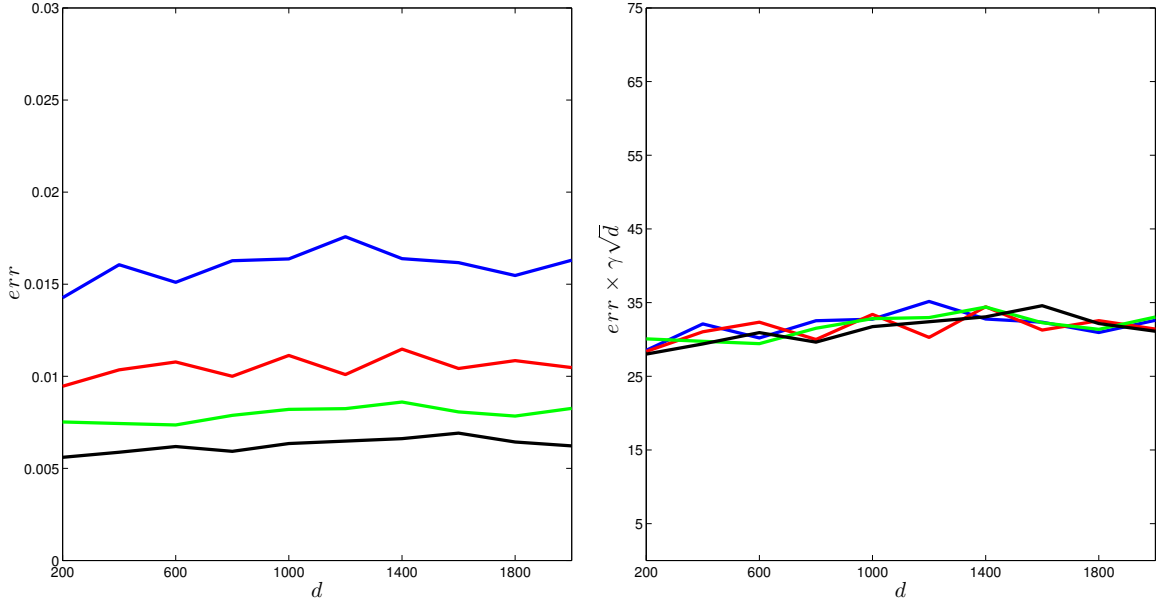}
    \caption{\textit{The left plot shows the perturbation error of eigenvectors against matrix size $d$ ranging from $200$ to $2000$, when $\gamma \sqrt{d}$ is kept fixed, with different values. The right plot shows the error multiplied by $\gamma \sqrt{d}$ against $d$. Blue lines represent $\gamma \sqrt{d} = 2000$; red lines $\gamma \sqrt{d} = 3000$; green lines $\gamma \sqrt{d} = 4000$; and black lines $\gamma \sqrt{d} = 5000$. We report the largest error over $100$ runs.}}
    \label{ptbFig2}
\end{figure}

\begin{figure}
    \centering
    \includegraphics[width=1\textwidth]{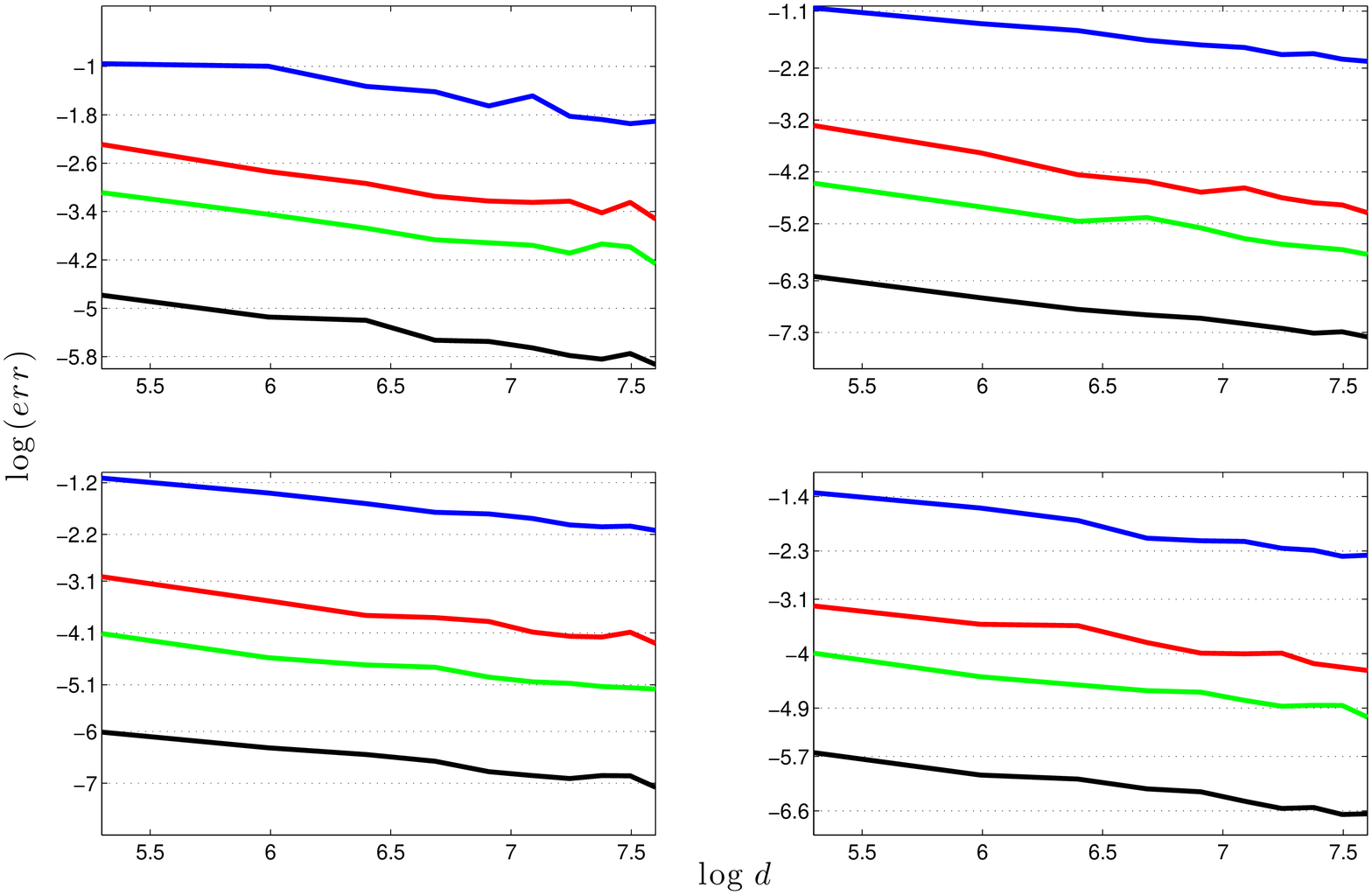}
    \caption{\textit{These plots show $\log(err)$ aginst $\log(d)$, with matrix size $d$ ranging from $200$ to $2000$ and different eigengap $\gamma$. The perturbation $E$ is generated from different ways. Top left: $L = 10, s = 3$; top right: $L = 0.6, s = 50$; bottom left: $L' = 1.5, \rho = 0.9$; bottom right: $L' = 7.5, \rho = 0.5$. The slopes are around $-0.5$. Blue lines represent $\gamma = 10$; red lines $\gamma = 50$; green lines $\gamma = 100$; and black lines $\gamma = 500$. We report the largest error over $100$ runs.}}
    \label{ptbFig3}
\end{figure}

In Figure \ref{ptbFig1} -- \ref{ptbFig3}, we report the largest  error based on $100$ runs. Figure \ref{ptbFig1} shows that the error decreases as $d$ increases (the left plot); and moreover, the logarithm of the error is linear in $\log(d)$, with a slope $-0.5$, that is, $err \propto 1/\sqrt{d}$ (the right plot). We can take the eigengap $\gamma$ into consideration and characterize the relationship in a more refined way. In Figure \ref{ptbFig2}, it is clear that $err$ almost falls on the same horizontal line for different configurations of $d$ and $\gamma$, with $\gamma\sqrt{d}$ fixed. The right panel clearly indicates that $err \times \gamma\sqrt{d}$ is a constant, and therefore $err \propto 1/(\gamma\sqrt{d})$. In Figure \ref{ptbFig3}, we find that the errors behave almost the same regardless of how $E$ is generated.
These simulation results provide stark evidence supporting the $\ell_{\infty}$ perturbation bound in Theorem \ref{thm::symindiv}.

\subsection{Simulation: robust covariance esitmation} \label{sec4.2}

We consider the performance of the generic POET procedure in robust covariance estimation in this subsection. Note that the procedure is flexible in employing any pilot estimators $\widehat\Sigma, \widehat\Lambda, \widehat V$ satisfying the conditions (\ref{suff1}) -- (\ref{suff3}) respectively.

We implemented the robust procedure with four different initial trios: (1) the sample covariance $\widehat{\Sigma}^S$ with its leading $r$ eigenvalues and eigenvectors as $\widehat\Lambda^S$ and $\widehat V^S$;
(2) the Huber's robust estimator $\widehat{\Sigma}^R$ given in (\ref{eqn:huber}) and its top $r$ eigen-structure estimators $\widehat\Lambda^R$ and $\widehat V^R$;
(3) the marginal Kendall's tau estimator $\widehat{\Sigma}^K$ with its corresponding $\widehat\Lambda^K$ and $\widehat V^K$;
(4) lastly, we use the spatial Kendall's tau estimator to estimate the leading eigenvectors instead of the marginal Kendall' tau, so $\widehat V^K$ in (3) is replaced with $\widetilde V^K$. We need to briefly review the two types of Kendall's tau estimators here, and specifically give the formula for $\widehat{\Sigma}^K$ and $\widetilde V^K$.

Kendall's tau correlation coefficient, for estimating pairwise comovement correlation, is defined as
\beq
\hat\tau_{jk} := \frac{2}{n(n-1)} \sum_{t < t'} \sgn((y_{tj} - y_{t'j})(y_{tk} - y_{t'k})) \,.
\eeq
Its population expectation is related to the Pearson correlation via the transform
$r_{jk} = \sin\Bigl(\frac{\pi}{2} \, E[\hat\tau_{jk}] \Bigr)$ for elliptical distributions (which are far too restrictive for high-dimensional applications). Then $\hat r_{jk} = \sin\Bigl(\frac{\pi}{2} \hat\tau_{jk} \Bigr)$ is a valid estimation for the Pearson correlation $r_{jk}$.
Letting $\widehat R = (\hat r_{jk})$  and $\widehat D = \diag(\sqrt{\widehat{\Sigma}_{11}^R}, \dots, \sqrt{\widehat{\Sigma}_{dd}^R})$ containing the robustly estimated standard deviations, we define the marginal Kendall's tau estimator as
\begin{equation}
\widehat{\Sigma}^K = \widehat D\, \widehat R\, \widehat D \,.
\end{equation}
In the above construction of $\widehat D$, we still use the robust variance estimates from $\widehat\Sigma^R$.

The spatial Kendall's tau estimator is a second-order U-statstic, defined as
\beq\label{eq::mkendall}
\widetilde{\Sigma}^K :=  \frac{2}{n(n-1)} \sum_{t < t'} \frac{(y_t- y_{t'})(y_t - y_{t'})^T}{\|y_t - y_{t'}\|_2^2} \,.
\eeq
Then $\widetilde V^S$ is constructed by the top $r$ eigenvectors of $\widetilde{\Sigma}^K$.
It has been shown by \cite{FanLiuWan17} that under elliptical distribution, $\widehat{\Sigma}^K$ and its top $r$ eigenvalues $\widehat\Lambda^K$ satisfy (\ref{suff1}) and (\ref{suff2}) while $\widetilde V^S$ suffices to conclude (\ref{suff3}). Hence Method (4) indeed provides good initial estimators if data are from elliptical distribution. However, since $\widehat{\Sigma}^K$ attains (\ref{suff1}) for elliptical distribution, by similar argument for deriving Proposition \ref{prop3.3} based on our $\ell_{\infty}$ pertubation bound, $\widehat V^K$ consisting of the leading eigenvectors of $\widehat{\Sigma}^K$ is also valid for the generic POET procedure. For more details about the two types of Kendall's tau, we refer the readers to \cite{FanKotNg90, ChoMar98, HanLiu14, FanLiuWan17} and references therein.

In summary, Method (1) is designed for the case of sub-Gaussian data; Method (3) and (4) work under the situation of elliptical distribution; while Method (2) is proposed in this paper for the general heavy-tailed case with bounded fourth moments without further distributional shape constraints.

We simulated $n$ samples of $(f_t^T, u_t^T)^T$ from two settings: (a) a multivariate t-distribution with covariance matrix diag$\{ I_r, 5 I_d \}$ and various degrees of freedom ($\nu = 3$ for very heavy tail, $\nu = 5$ for medium heavy tail and $\nu = \infty$ for Gaussian tail), which is one example of the elliptical distribution \citep{FanKotNg90}; (b) an element-wise iid one-dimensional t distribution with the same covariance matrix and degrees of freedom $\nu = 3, 5$ and $\infty$, which is a non-elliptical heavy-tailed distribution.

Each row of coefficient matrix $B$ is independently sampled from a standard normal distribution, so that with high probability, the pervasiveness condition holds with $\|B\|_{\max} = O(\sqrt{\log d})$. The data is then generated by $y_t = B f_t + u_t$ and the true population covariance matrix is $\Sigma = B B^T + 5 I_d$.

\begin{figure}
\begin{center}
\includegraphics[scale = 0.85]{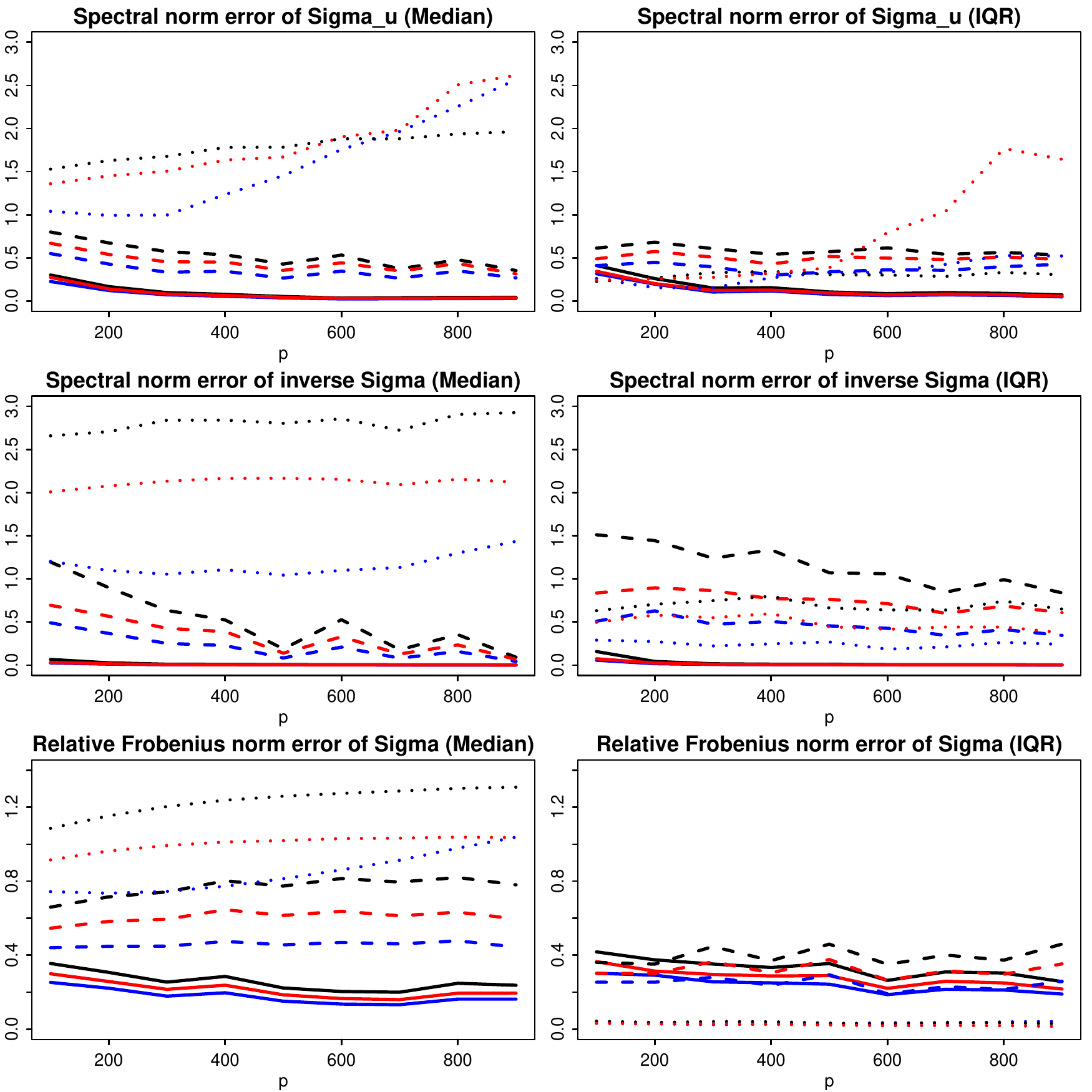}
\end{center}
\caption{\textit{Error ratios of robust estimates against varying dimension. Blue lines represent errors of Method (2) over Method (1) under different norms; black lines errors of Method (3) over Method (1); red lines errors of Method (4) over Method (1). $(f_t^T, u_t^T)$ is generated by multivariate t-distribution with $df = 3$ (solid), $5$ (dashed) and $\infty$ (dotted). The median errors and their IQR's (interquartile range) over $100$ simulations are reported. }}
\label{figure::exp1}
\end{figure}

For $d$ running from $200$ to $900$ and $n = d/2$, we calculated errors of the four robust estimators in different norms.
The tuning for $\alpha$ in minimization (\ref{eqn:huber}) is discussed more throughly in \cite{FanWanZho17}.
For the thresholding parameter, we used $\tau = 2 \sqrt{\log d/n}$.
The estimation errors are gauged in the following norms: $\| \widehat{\Sigma}_u^{\top} - \Sigma_u \|$, $\| (\widehat\Sigma^{\top})^{-1} - {\Sigma}^{-1} \|$ and $\|\widehat{\Sigma}^{\top} - \Sigma\|_{\Sigma}$ as shown in Theorem \ref{suff}.
The two different settings are separately plotted in Figures \ref{figure::exp1} and \ref{figure::exp2}. The estimation errors of applying sample covariance matrix $\widehat{\Sigma}^S$ in Method (1) are used as the baseline for comparison. For example, if relative Frobenius norm is used to measure performance, $\|(\widehat \Sigma^{\top})^{(k)} - {\Sigma}\|_{\Sigma}/\|(\widehat{\Sigma}^{\top})^{(1)} - {\Sigma}\|_{\Sigma}$ will be depicted for $k = 2, 3, 4$, where $(\widehat \Sigma^{\top})^{(k)}$ are generic POET estimators based on Method ($k$). Therefore if the ratio curve moves below $1$, the method is better than naive sample estimator \citep{FanLiaMin13} and vice versa. The more it gets below $1$, the more robust the procedure is against heavy-tailed randomness.

The first setting (Figure \ref{figure::exp1}) represents a heavy-tailed elliptical distribution, where we expect Methods (2), (3), (4) all outperform the POET estimator based on the sample covariance, i.e. Method (1), especially in the presence of extremely heavy tails (solid lines for $\nu = 3$). As expected, all three curves under various measures show error ratios visibly smaller than $1$. On the other hand, if data are indeed Gaussian (dotted line for $\nu = \infty$), Method (1) has better behavior under most measures (error ratios are greater than $1$). Nevertheless, our robust Method (2) still performs comparably well with Method (1), whereas the median error ratios for the two Kendall's tau methods are much worse. In addition, the IQR (interquartile range) plots reveal that Method (2) is indeed more stable than two Kendall's tau Methods (3) and (4). It is also noteworthy that Method (4), which leverages the advantage of spatial Kendall's tau, performs more robustly than Method (3), which solely base its estimation of the eigen-structure on marginal Kendall's tau.

 \begin{figure}
\begin{center}
\includegraphics[scale = 0.85]{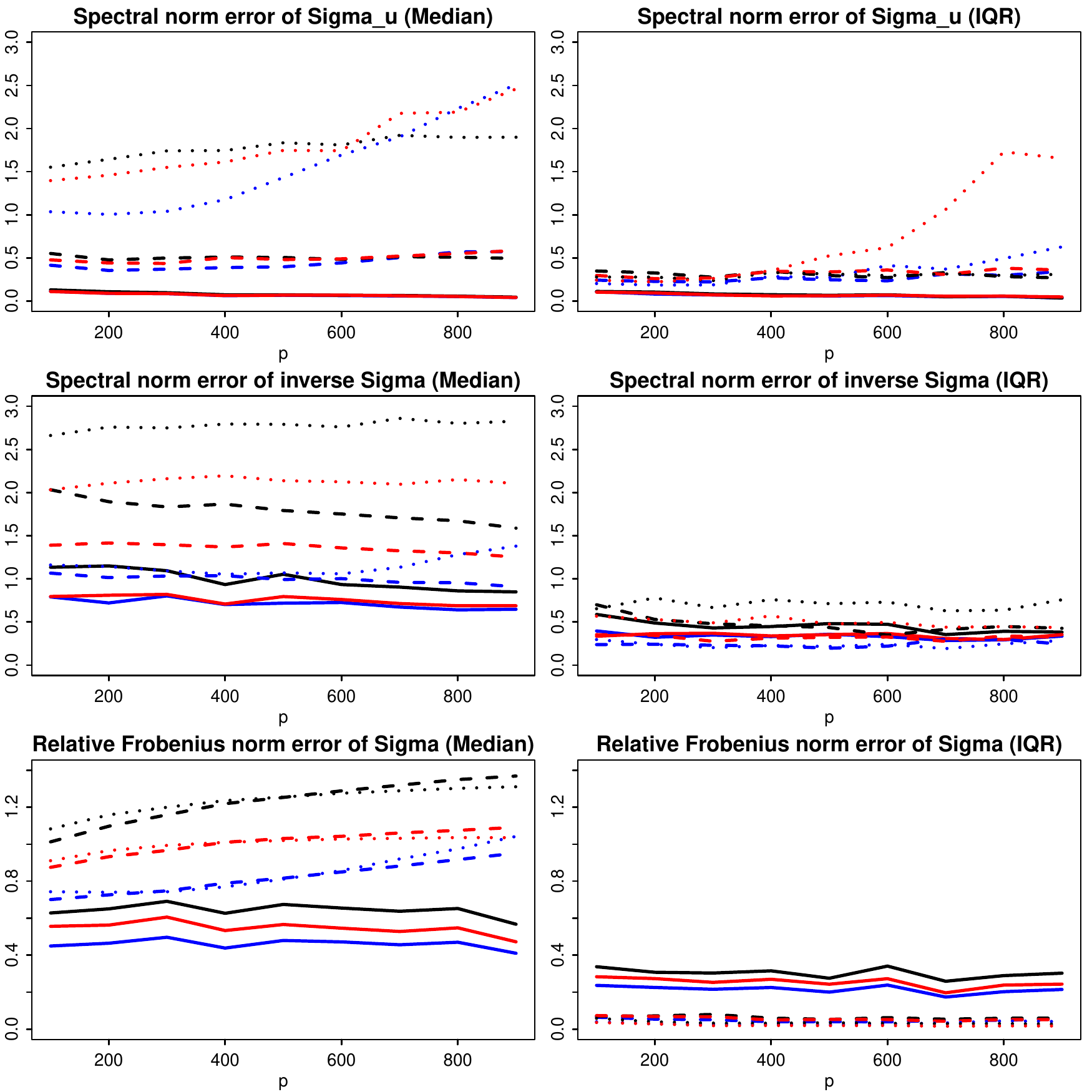}
\end{center}
\caption{\textit{Error ratios of robust estimates against varying dimension. Blue lines represent errors of Method (2) over Method (1) under different norms; black lines errors of Method (3) over Method (1); red lines errors of Method (4) over Method (1). $(f_t^T, u_t^T)$ is generated by element-wise iid t-distribution with $df = 3$ (solid), $5$ (dashed) and $\infty$ (dotted). The median errors and their IQR's (interquartile range) over $100$ simulations are reported.}}
\label{figure::exp2}
\end{figure}

The second setting (Figure \ref{figure::exp2}) provides an example of non-elliptical distributed data. We can see that the performance of the general robust Method (2) dominates the other three methods, which verifies the benefit of robust estimation for a general heavy-tailed distribution. Note that Kendall's tau methods do not apply to distributions outside the elliptical family, excluding even the element-wise iid $t$ distribution in this setting. Nonetheless, even in the first setting where the data are indeed elliptical, with proper tuning, the proposed robust method can still outperform Kendall's tau by a clear margin.

\section{Proof Organization of Main Theorems}\label{sec::org}
\subsection{Symmetric Case}
For shorthand, we write $\tau  = \| E \|_{\infty}$, and $\kappa = \sqrt{d}\,\| E V \|_{\max}$. An obvious bound for $\kappa$ is $\kappa \le \sqrt{r \mu}\, \tau$ (by Cauchy-Schwarz inequality). We will use these notations throughout this subsection.

Recall the spectral decomposition of $A$ in (\ref{eq:specA}). Expressing $E $ in terms of the column vectors of $V$ and $V_\bot$, which form an orthogonal basis in $\mathbb{R}^n$, we write
\begin{equation}\label{def::Eblock}
[V,V_\bot]^T E [V,V_\bot] =:  \left( \begin{array}{cc}
E_{11} & E_{12} \\ E_{21} & E_{22} \end{array} \right)\,.
\end{equation}
Note that $E_{12} = E_{21}^T$ since $E$ is symmetric. Conceptually, the perturbation results in a rotation of $[V,V_\bot]$, and we write a candidate orthogonal basis as follows:
\begin{equation}\label{def::Vbar}
\overline{V} := (V + V_\bot Q)(I_r + Q^T Q)^{-1/2}, \qquad \overline{V}_\bot := (V_\bot - VQ^T) (I_{d-r} + QQ^T)^{-1/2},
\end{equation}
where $Q \in \mathbb{R}^{(d-r) \times r}$ is to be determined. It is straightforward to check that $[\overline{V}, \overline{V}_\bot]$ is an orthogonal matrix. We will choose $Q$ in a way such that $(\overline{V}, \overline{V}_\bot)^T \widetilde{A} (\overline{V}, \overline{V}_\bot)$ is a block diagonal matrix, i.e., $\overline{V}_\bot^T \widetilde{A} \overline{V} = 0$. Substituting (\ref{def::Vbar}) and simplifying the equation, we obtain
\begin{equation}\label{eqn:1}
 Q(\Lambda_1 + E_{11}) - (\Lambda_2 + E_{22}) Q = E_{21} - QE_{12}Q.
\end{equation}
The approach of studying perturbation through a quadratic equation is known (see \cite{Ste90} for example). Yet, to the best of our knowledge, existing results study perturbation under orthogonal-invariant norms (or unitary-invariant norms in the complex case), which includes a family of matrix operator norms and Frobenius norm, but excludes the matrix max-norm. The advantages of orthogonal-invariant norms are pronounced: such norms of a symmetric matrix only depend on its eigenvalues regardless of eigenvectors; moreover, with suitable normalization they are \textit{consistent} in the sense $\| A B \| \le \| A \| \cdot \| B \|$. See \cite{Ste90} for a clear exposition.

The max-norm, however, does not possess these important properties. An imminent issue is that it is not clear how to relate $Q$ to $V_\bot Q$, which will appear in (\ref{eqn:1}) after expanding $E$ according to (\ref{def::Eblock}), and which we want to control. Our approach here is to study $\overline{Q} := V_\bot Q$ directly through a transformed quadratic equation, obtained by left multiplying $V_\bot$ to (\ref{eqn:1}). Denote $H = V_\bot E_{21}, \overline{Q} = V_\bot Q, \overline{L}_1 = \Lambda_1 + E_{11}, \overline{L}_2 =  V_\bot (\Lambda_2 + E_{22}) V_\bot^T$. If we can find an appropriate matrix $\overline{Q}$ with $\overline{Q} = V_\bot Q$, and it satisfies the quadratic equation
\begin{equation}\label{eqn:2}
\overline{Q}\, \overline{L}_1 - \overline{L}_2 \overline{Q} = H - \overline{Q}H^T\overline{Q},
\end{equation}
then $Q$ also satisfies the quadratic equation (\ref{eqn:1}). This is because multiplying both sides of (\ref{eqn:2}) by $V_\bot^T$ yields (\ref{eqn:1}), and thus any solution $\overline{Q}$ to (\ref{eqn:2}) with the form $\overline{Q} = V_\bot Q$ must result in a solution $Q$ to (\ref{eqn:1}).

Once we have such $\overline{Q}$ (or equivalently $Q$), then $(\overline{V}, \overline{V}_\bot)^T \widetilde{A} (\overline{V}, \overline{V}_\bot)$ is a block diagonal matrix, and the span of column vectors of $\overline{V}$ is a candidate space of the span of first $r$ eigenvectors, namely $\mathrm{span}\{\widetilde{v}_1,\ldots, \widetilde{v}_r \}$. We will verify the two spaces are identical in Lemma \ref{lem::match}. Before stating that lemma, we first provide bounds on $\| \overline{Q} \|_{\max}$ and $ \| \overline{V} - V \|_{\max}$.
\begin{lem}\label{lem::Qbar}
Suppose $|\lambda_r| - \varepsilon > 4r \mu(\tau + 2r \kappa)$. Then, there exists a matrix $Q \in \mathbb{R}^{(d - r) \times r}$ such that $\overline{Q} = V_\bot Q \in \mathbb{R}^{d \times r}$ is a solution to the quadratic equation (\ref{eqn:2}), and $\overline{Q}$ satisfies $\| \overline{Q} \|_{\max} \le \omega / \sqrt{d}$. Moreover, if $r \omega < 1/2$, the matrix $\overline{V}$ defined in (\ref{def::Vbar}) satisfies
\begin{equation}\label{ineqn::Vbarbound}
\| \overline{V} - V \|_{\max} \le  2\sqrt{\mu}\, \omega r  / \sqrt{d}\,.
\end{equation}
Here, $\omega$ is defined as $ \omega = 8(1+r\mu) \kappa /  (|\lambda_r| - \varepsilon) $.
\end{lem}
The second claim of the lemma (i.e., the bound (\ref{ineqn::Vbarbound})) is relatively easy to prove once the first claim (i.e., the bound on $\| \overline{Q} \|_{\max}$) is proved. To understand this, note that we can rewrite $\overline{V}$ as $\overline{V} = (V + \overline{Q})(I_r + \overline{Q}^T \overline{Q})^{-1/2}$, and $\| \overline{Q}^T \overline{Q} \|_{\max}$ can be controlled by a trivial inequality $\| \overline{Q}^T \overline{Q} \|_{\max} \le d\| \overline{Q} \|_{\max}^2 \le w^2$. To prove the first claim, we construct a sequence of matrices through recursion that converges to the fixed point $\overline{Q}$, which is a solution to the quadratic equation (\ref{eqn:2}). For all iterates of matrices, we prove a uniform max-norm bound, which leads to a max-bound on $\| \overline{Q} \|_{\max}$ by continuity. To be specific, we initialize $\overline{Q}^0 = 0$, and given $\overline{Q}^t$, we solve a linear equation:
\begin{equation}\label{eqn::linear}
\overline{Q}\, \overline{L}_1 - \overline{L}_2 \overline{Q} = H - \overline{Q}^tH^T\overline{Q}^t,
\end{equation}
and the solution is defined as $\overline{Q}^{t+1}$. Under some conditions, the iterate $\overline{Q}^t$ converges to a limit $\overline{Q}$, which is a solution to (\ref{eqn:2}). The next general lemma captures this idea. It follows from \cite{Ste90} with minor adaptations.
\begin{lem}\label{lem::borrow}
Let $T$ be a bounded linear operator on a Banach space $\mathcal{B}$ equipped with a norm $\| \cdot \|$. Assume that $T$ has a bounded inverse, and define $\beta = \| T^{-1} \|^{-1}$. Let $\varphi: \mathcal{B} \to \mathcal{B}$ be a map that satisfies
\begin{equation}\label{ineqn::phi}
\| \varphi(x) \| \le \eta \|x\|^2, \qquad \text{and} \qquad \| \varphi(x) - \varphi(y) \| \le 2 \eta \max \{ \|x\|, \|y\| \} \|x - y\|
\end{equation}
for some $\eta \ge 0$. Suppose that $\mathcal{B}_0$ is a closed subspace of $\mathcal{B}$ such that $T^{-1}(\mathcal{B}_0) \subseteq \mathcal{B}_0$ and $\varphi(\mathcal{B}_0) \subseteq \mathcal{B}_0$. Suppose $y \in \mathcal{B}_0$ that satisfies $4 \eta \|y \| < \beta^2$. Then, the sequence initialized with $x_0 = 0$ and iterated through
\begin{equation} \label{eqn::seq}
x_{k+1} = T^{-1}(y + \varphi(x_k)), \quad k \ge 0
\end{equation}
converges to a solution $x^{\star}$ to $Tx = y + \varphi(x)$. Moreover, we have $x^{\star} \subseteq \mathcal{B}_0$, and $\| x^{\star} \| \le 2 \| y \| / \beta$.
\end{lem}

To apply this lemma to the equation (\ref{eqn:2}), we view $\mathcal{B}$ as the space of matrices $\mathbb{R}^{d \times r}$ with the max-norm $\| \cdot \|_{\max}$, and $\mathcal{B}_0$ as the subspace of matrices of the form $V_\bot Q$ where $Q \in \mathbb{R}^{(d-r) \times r}$. The linear operator $T$ is set to be the $T (\overline{Q}) = \overline{Q}\, \overline{L}_1 - \overline{L}_2 \overline{Q}$, and the map $\varphi$ is set to be the quadratic function $\varphi( \overline{Q} ) = -\overline{Q}H^T\overline{Q}$. Roughly speaking, under the assumption of Lemma \ref{lem::borrow}, the nonlinear effect caused by $\varphi$ is weak compared with the linear operator $T$. Therefore, it is crucial to show $T$ is invertible,  i.e. to give a good lower bound on $\|T^{-1}\|_{\max}^{-1} = \inf_{\|\overline{Q}\|_{\max}  = 1} \|T(\overline{Q}) \|_{\max}$. Since the norm is not orthogonal-invariant, a subtle issue arises when $A$ is not of exact low rank, which will be discussed at the end of the subsection.

If there is no perturbation (i.e., $E = 0$), all the iterates $\overline{Q}^t$ are simply $0$, so $\overline{V}$ is identical to $V$. If the perturbation is not too large, the next lemma shows that the column vectors of $\overline{V}$ span the same space as $\mathrm{span}\{ \widetilde{v}_1,\ldots, \widetilde{v}_r \}$.

In other words, with a suitable orthogonal matrix $R$, the columns of $\overline{V} R$ are $\widetilde{v}_1,\ldots, \widetilde{v}_r$.
\begin{lem}\label{lem::match}
Suppose $|\lambda_r| - \varepsilon > \max\{3 \tau, 64(1+r\mu) r^{3/2} \mu^{1/2} \kappa  \}$. Then, there exists an orthogonal matrix $R \in \mathbb{R}^{r \times r}$ such that the column vectors of $\overline{V} R$ are $\widetilde{v}_1,\ldots, \widetilde{v}_r$.
\end{lem}

\begin{proof}[{\bf Proof of Theorem \ref{thm::symbulk}}]
It is easy to check that under the assumption of Theorem \ref{thm::symbulk}, the conditions required in Lemma \ref{lem::Qbar} and Lemma \ref{lem::match} are satisfied. Hence, the two lemmas imply Theorem \ref{thm::symbulk}.
\end{proof}

To study the perturbation of individual eigenvectors, we assume, in addition to the condition on $|\lambda_r|$, that $\lambda_1,\ldots, \lambda_r$ satisfy a uniform gap, (namely $\delta > \| E \|_2$). This additional assumption is necessary, because otherwise, the perturbation may lead to a change of relative order of eigenvalues, and we may be unable to match eigenvectors from the order of eigenvalues. Suppose $R \in \mathbb{R}^{r \times r}$ is an orthogonal matrix such that $\overline{V}R$ are eigenvectors of $\widetilde{A}$. Now, under the assumption of of Theorem \ref{thm::symbulk}, the column vectors of $\widetilde{V}$ and $\overline{V}R$ are identical up to sign, so we can rewrite the difference $\widetilde{V} - V$ as
\begin{equation}\label{eqn::Vdecomp}
\widetilde{V} - V = \overline{V}(R - I_r) + (\overline{V} -V).
\end{equation}
We already provided a bound on $\| \overline{V} -V \|_{\max}$ in Lemma \ref{lem::Qbar}. By the triangular inequality, we can derive a bound on $\| \overline{V} \|_{\max}$. If we can prove a bound on $\| R - I_r \|_{\max}$, it will finally leads to a bound on $\| \widetilde{V} - V \|_{\max}$. In order to do so, we use the Davis-Kahan theorem to obtain an bound on $\langle \widetilde{v}_i, v_i \rangle$ for all $i \in [r]$. This will lead to a max-norm bound on $R - I_r$ (with the price of potentially increasing the bound by a factor of $r$). The details about the proof of Theorem \ref{thm::symindiv} are in the appendix.

We remark that, we assume conditions on $|\lambda_r| - \epsilon$ in Theorem \ref{thm::symbulk} and Theorem \ref{thm::symindiv}, which are only useful in cases where $|\lambda_r| > \| A - A_r \|_\infty$. Ideally, we would like to have results with assumptions only involving $\lambda_r$ and $\lambda_{r+1}$, since Davis-Kahan theorem only requires a gap in neighboring eigenvalues. Unfortunately, unlike orthogonal-invariant norms that only depend on the eigenvalues of a matrix, the max-norm $\| \cdot \|_{\max}$ is not orthogonal-invariant, and thus it also depends on the eigenvectors of a matrix. For this reason, it is not clear whether we could obtain a lower bound on $\|T^{-1}\|_{\max}^{-1}$ using only the eigenvalues $\lambda_r$ and $\lambda_{r+1}$ so that we could apply Lemma \ref{lem::borrow}. The analysis appears to be difficult if we do not have a bound on $\| T^{-1} \|_{\max}^{-1}$, considering that even in the analysis of linear equations, we also need invertibility, condition numbers, etc.

\subsection{Asymmetric Case}\label{sec:orgasym}

Let $A^d, E^d$ be $d_1 + d_2$ square matrices defined as
\begin{equation*}
A^d:= \left( \begin{array}{cc} 0 & A \\ A^T & 0 \end{array} \right), \qquad E^d:= \left( \begin{array}{cc} 0 & E \\ E^T & 0 \end{array} \right).
\end{equation*}
Also denote $\widetilde{A}^d := A^d + E^d$. This augmentation of an asymmetric matrix into a symmetric one is called Hermitian dilation. Here the superscript $d$ means the Hermitian dilation. We also use this notation to denote quantities corresponding to $A^d$ and $\widetilde A^d$.

An important observation is that
\begin{equation*}
\left( \begin{array}{cc} 0 & A \\ A^T & 0 \end{array} \right) \left( \begin{array}{c} u_i \\ \pm\, v_i \end{array} \right) = \pm \,  \sigma_i  \left( \begin{array}{c} u_i \\ \pm\, v_i \end{array} \right).
\end{equation*}
From this identity, we know that $A^d$ have nonzero eigenvalues $\pm\, \sigma_i$ where $1 \le i \le \mathrm{rank}(A)$, and its corresponding eigenvectors are $( u_i^T, \pm\, v_i^T)^T$. For a given $r$, we stack these (normalized) eigenvectors with indices $i \in [r]$ into a matrix $V^d \in \mathbb{R}^{(d_1 + d_2) \times 2r}$:
\begin{equation*}
V^d := \frac{1}{\sqrt{2}} \left( \begin{array}{cc} U & U \\ V & -V \end{array} \right) \,.
\end{equation*}
Through the augmented matrices, we can transfer eigenvector results for symmetric matrices to singular vectors of asymmetric matrices. However, we cannot directly invoke the results proved for symmetric matrices, due to an issue about the coherence of $V^d$: when $d_1$ and $d_2$ are not comparable, the coherence $\mu(V^d)$ can be very large even when $\mu(V)$ and $\mu(U)$ are bounded. To understand this, consider the case where $r=1$, $d_1 \gg d_2$, and all entries of $U$ are $O(1/\sqrt{d_1})$, and all entries of $V$ are $O(1/\sqrt{d_2})$. Then, the coherences $\mu(U)$ and $\mu(V)$ are $O(1)$, but $\mu(V^d) = O((d_1 + d_2)/d_2) \gg 1$.

This unpleasant issue about the coherence, nevertheless, can be tackled if we consider a different matrix norm. In order to deal with the different scales of $d_1$ and $d_2$, we define the weighted max-norm for any matrix $M$ with $d_1 +d_2$ rows as follows:
\begin{equation}
\|M\|_{w} := \Big\| \left( \begin{array}{cc} \sqrt{d_1} I_{d_1} & 0 \\ 0 & \sqrt{d_2} I_{d_2} \end{array} \right) M\Big\|_{\max}\,.
\end{equation}
In other words, we rescale the top $d_1$ rows of $M$ by a factor of $\sqrt{d}_1$, and rescale the bottom $d_2$ rows by $\sqrt{d}_2$. This weighted norm serves to balance the potential different scales of $d_1$ and $d_2$.

The proofs of theorems in Section \ref{sec::asymm} will be almost the same with those in the symmetric case, with the major difference being the new matrix norm. Because the derivation is slightly repetitive, we will provide concise proofs in the appendix . Similar to the decomposition in (\ref{sec::sym}),
\begin{equation*}
A^d  =  \left( \begin{array}{cc} 0 & A_r \\ A_r^T & 0 \end{array}\right) + \left( \begin{array}{cc} 0 & A-A_r \\ A^T-A_r^T & 0 \end{array}\right) =: A_r^d + (A^d - A_r^d ) ,
\end{equation*}
where $A_r^d$ is has rank $2r$. Equivalently,
\begin{equation*}
A_r^d = \sum_{i=1}^r \sigma_i ( u_i^T,  v_i^T)^T ( u_i^T,  v_i^T) - \sum_{i=1}^r \sigma_i ( u_i^T,  -v_i^T)^T ( u_i^T,  -v_i^T).
\end{equation*}
Analogously, we define notations in (\ref{def::Vbar})--(\ref{eqn:2}) and use $d$ in the superscript to signify that they are augmented through Hermitian dilation. It is worthwhile to note that $\Lambda_1^d = \diag\{\sigma_1,\ldots,\sigma_r,-\sigma_r, \ldots, -\sigma_1\}$, and that $\min \{ |\pm \sigma_i| : i \in [r]\} = \sigma_r$ (a similar quantity as $|\lambda_r|$). Recall $\mu_0 = \mu(U) \vee \mu(V)$, $\tau_0 = \sqrt{d_1/d_2} \|E\|_{\infty} \vee \sqrt{d_2/d_1} \| E \|_{1}$ and $\varepsilon_0 = \sqrt{d_1/d_2} \|A-A_r\|_{\infty} \vee \sqrt{d_2/d_1} \| A-A_r \|_{1}$. In the proof, we will also use $\kappa_0 = \max \{ \sqrt{d_1}\, \| E V \|_{\max}, \sqrt{d_2}\, \| E^T U \|_{\max} \}$, which is a quantity similar to $\kappa$.

The next key lemma, which is parallel to Lemma \ref{lem::Qbar}, provides a bound on the solution $\overline{Q}^d$ to the quadratic equation
\begin{equation}\label{eqn:2_2}
\overline{Q}^d\, \overline{L}_1^d - \overline{L}_2^d \overline{Q}^d = H^d - \overline{Q}^d (H^d)^T\overline{Q}^d.
\end{equation}

\begin{lem}\label{lem::Qbar2}
Suppose $\sigma_r - \varepsilon_0 > 16r\mu_0 (\tau_0 + r\kappa_0)$. Then, there exists a matrix $Q^d \in \mathbb{R}^{(d_1 + d_2 - 2r) \times 2r}$ such that $\overline{Q}^d = V_\bot^d Q^d \in \mathbb{R}^{(d_1 + d_2) \times 2r}$ is a solution to the quadratic equation (\ref{eqn:2_2}), and $\overline{Q}^d$ satisfies $\| \overline{Q}^d \|_{w} \le \omega_0$. Moreover, if $r \omega_0 < 1/2$, the matrix $\overline{V}^d$ defined in (\ref{def::Vbar}) satisfies
\begin{equation}\label{ineqn::Vbarbound2}
\| \overline{V}^d - V^d \|_{w} \le  6 \sqrt{\mu_0} \, r \omega_0\,.
\end{equation}
Here, $\omega_0$ is defined as $ \omega_0 = 8(1+r \mu_0 ) \kappa_0 /3( \sigma_r  -  \varepsilon_0) $.
\end{lem}

In this lemma, the bound \eqref{ineqn::Vbarbound2} bears a similar form to \eqref{ineqn::Vbarbound}: if we consider the max-norm, the first $d_1$ rows of $\overline{V}^d - V^d$ correspond to the left singular vectors $u_i$'s, and they scale with $1/\sqrt{d_1}$; and the last $d_2$ rows correspond to the right singular vectors $v_i$'s, which scale with $1/\sqrt{d_2}$. Clearly, the weighted max-norm $\| \cdot \|_w$ indeed helps to balance the two dimensions.

\appendix

\section{Proofs for Section \ref{sec::sym}}
Denote the column span of a matrix $M$ by $\mathrm{span}(M)$. Suppose two matrices $M_1,M_2 \in \mathbb{R}^{n \times m}$ ($m \le n$) have orthonormal column vectors. It is known that \citep{Ste90}
\begin{equation}\label{ineqn:equiv}
d(M_1,M_2) := \| M_1 M_1^T - M_2 M_2^T \|_2 = \| \sin \Theta(M_1,M_2) \|_2.
\end{equation}
where $\Theta(M_1,M_2)$ are the canonical angles between $\mathrm{span}(M_1)$ and $\mathrm{span}(M_2)$.
Recall the notations defined in (\ref{def::Eblock}), and also recall $\kappa = \sqrt{d} \|EV\|_{\max}$, $\Lambda_1 = \text{diag}\{ \lambda_1, \ldots, \lambda_r \}$, $\Lambda_2 = \text{diag}\{ \lambda_{r+1}, \ldots, \lambda_n \}$, $ \overline{L}_1 = \Lambda_1 + E_{11}$, $\overline{L}_2 = V_\bot (\Lambda_2 + E_{22}) V_\bot^T$ and $H = V_\bot E_{21}$. The first lemma bounds $\| H \|_{\max}$.
\begin{lem}\label{lem::H}
We have the following bound on $\| H \|_{\max}$:
\begin{equation*}
\| H \|_{\max} \le (1 + r \mu ) \kappa / \sqrt{d}.
\end{equation*}
\end{lem}
\begin{proof}[Proof of Lemma \ref{lem::H}]
Using the definition $E_{21} = V_\bot^T E V$ in (\ref{def::Eblock}), we can write $H =  V_\bot V_\bot^T E V$. Since the columns of $V$ and $V_\bot$ form an orthogonal basis in $\mathbb{R}^d$, clearly
\begin{equation}\label{eqn::VZorth}
V V^T + V_\bot V_\bot^T = I_d\,.
\end{equation}
By Cauchy-Schwarz inequality and the definition of $\mu$, for any $i, j \in [d]$,
\begin{equation*}
|( V V^T )_{ij}| = \sum_{k=1}^r |V_{ik} V_{jk}| \le\big(\sum_{k=1}^r V_{ik}^2 \big)^{1/2} \cdot \big( \sum_{k=1}^r V_{jk}^2 \big)^{1/2} \le \frac{r \mu}{d} \,.
\end{equation*}
Using the identity (\ref{eqn::VZorth}) and the above inequality, we derive
\begin{align*}
\| H \|_{\max} & \le  \|EV\|_{\max} + \|VV^T EV\|_{\max} \\
& \le (1 + d \|VV^T\|_{\max} ) \, \|EV\|_{\max}  \le (1 + r \mu ) \| EV\|_{\max}\,,
\end{align*}
which completes the proof.
\end{proof}

\begin{lem}\label{lem::linop}
If $|\lambda_r| > \kappa r \sqrt{\mu}$, then $\overline{L}_1$ is an invertible matrix. Furthermore,
\begin{equation}
\inf_{\|Q_0\|_{\max} = 1} \| Q_0 \overline{L}_1 - \overline{L}_2 Q_0 \|_{\max} \ge  |\lambda_r| - 3r\mu(\tau + r\kappa)   - \varepsilon \,,
\end{equation}
where $Q_0$ is an $d \times r$ matrix.
\end{lem}
\begin{proof}[Proof of Lemma \ref{lem::linop}]
Let $Q_0$ be any $d \times r$ matrix with $\| Q_0 \|_{\max} = 1$. Note
\begin{equation*}
Q_0 \overline{L}_1 - \overline{L}_2 Q_0 = Q_0 \Lambda_1 + Q_0 E_{11} - \overline{L}_2 Q_0.
\end{equation*}
We will derive upper bounds on $Q_0 E_{11}$ and $\overline{L}_2 Q_0$, and a lower bound on $Q_0 \Lambda_1$. Since $E_{11} = V^T E V$ by definition, we expand $Q_0 E_{11}$ and use a trivial inequality to derive
\begin{equation}\label{ineqn::Q0E11}
\| Q_0 E_{11} \|_{\max} \le d \,\| Q_0 V^T \|_{\max} \| EV \|_{\max}\,.
\end{equation}
By Cauchy-Schwarz inequality and the definition of $\mu$  in (\ref{def::incoherence}), for $i,j \in [d]$,
\begin{equation*}
| (Q_0 V^T )_{ij} |  \le \sum_{k=1}^r |(Q_0)_{ik} V_{jk}|  \le \big(\sum_{k=1}^r (Q_0)_{ik}^2 \big)^{1/2} \, \big( \sum_{k=1}^r V_{jk}^2 \big)^{1/2}  \le \sqrt{r}  \cdot \sqrt{\frac{r \mu}{d}} \,,
\end{equation*}
Substituting $ \| EV \|_{\max} = \kappa / \sqrt{d}$ into (\ref{ineqn::Q0E11}), we obtain an upper bound:
\begin{equation}\label{ineqn::Q0E11-2}
\| Q_0 E_{11} \|_{\max} \le \kappa r \sqrt{\mu} \,.
\end{equation}
To bound  $\overline{L}_2 Q_0 = (V_\bot E_{22} V_\bot^T + (A - A_r))Q_0$, we use the identity (\ref{eqn::VZorth}) and write
\begin{equation*}
V_\bot E_{22} V_\bot^T Q_0= V_\bot V_\bot^T E V_\bot V_\bot^T Q_0 = (I_d - VV^T) E (I_d - VV^T) Q_0\,.
\end{equation*}
Using two trivial inequalities $ \| E Q_0 \|_{\max} \le \| E \|_{\infty} \| Q_0 \|_{\max} =  \| E \|_{\infty}$ and $ \| V^T Q_0 \|_{\max} \le \| V^T \|_{\infty} \| Q_0 \|_{\max} \le \sqrt{d}$, we have
\begin{align*}
\| E(I_d - VV^T)Q_0 \|_{\max} & \le \| EQ_0\|_{\max} + r \|EV \|_{\max} \| V^T Q_0 \|_{\max} \\
& \le  \| E \|_{\infty} + r  \sqrt{d} \, \| EV \|_{\max} = \tau + r \kappa \,.
\end{align*}
In the proof of Lemma \ref{lem::H}, we showed $\| VV^T \|_{\max} \le r \mu /d$. Thus,
\begin{equation*}
\| V_\bot E_{22} V_\bot^T Q_0\|_{\max} \le (1 + d \,\| VV^T \|_{\max}) \cdot \| E(I_d - VV^T)Q_0 \|_{\max} \le  (1+r\mu)(\tau + r \kappa) \,.
\end{equation*}
Moreover, $\| (A-A_r)Q_0 \|_{\max} \le \| A - A_r \|_\infty \| Q_0 \|_{\max} = \varepsilon$. Combining the two bounds,
\begin{equation}\label{ineqn::L2Q0}
\| \overline{L}_2 Q_0 \|_{\max} \le  (1+r\mu)(\tau + r \kappa) + \varepsilon.
\end{equation}
It is straightforward to obtain a lower bound on $\|Q_0 \Lambda_1\|_{\max}$: since there is an entry of $Q_0$, say $(Q_0)_{ij}$, that has an absolute value of $1$, we have
\begin{equation}\label{ineqn::Q0L1}
 \| Q_0 \Lambda_1 \|_{\max} \ge | (Q_0)_{ij} \lambda_j | \ge |\lambda_r| .
\end{equation}
To show $\overline{L}_1$ is invertible, we use (\ref{ineqn::Q0E11}) and (\ref{ineqn::Q0L1}) to obtain
\begin{equation*}
\|Q_0 \overline{L}_1 \|_{\max}  \ge \| Q_0 \Lambda_1 \|_{\max} - \| Q_0 E_{11} \|_{\max} \ge |\lambda_r| - \kappa r \sqrt{\mu}\,.
\end{equation*}
When $|\lambda_r| - \kappa r \sqrt{\mu} > 0$, $\overline{L}_1$ must have full rank, because otherwise we can choose an appropriate $Q_0$ in the null space of $\overline{L}_1^T$ so that $ Q_0 \overline{L}_1 = 0$, which is a contradiction. To prove the second claim of the lemma, we combine the lower bound (\ref{ineqn::Q0L1}) with upper bounds (\ref{ineqn::Q0E11-2}) and (\ref{ineqn::L2Q0}) to derive
\begin{align*}
\| Q_0 \overline{L}_1 - \overline{L}_2 Q_0 \|_{\max} &\ge \|Q_0L_1\|_{\max} - \| Q_0 E_{11} \|_{\max} - \| \overline{L}_2 Q_0 \|_{\max} \\
&\ge |\lambda_r| - \kappa r \sqrt{\mu} - (1+r\mu)(\tau + r \kappa) - \varepsilon \\
&\ge |\lambda_r| - 3r\mu(\tau + r\kappa) - \varepsilon \,,
\end{align*}
which is exactly the desired inequality.
\end{proof}

Next we prove Lemma \ref{lem::borrow}. This lemma follows from \cite{Ste90}, with minor changes that involves $\mathcal{B}_0$. We provide a proof for the sake of completeness.

\begin{proof}[{\bf Proof of Lemma \ref{lem::borrow}}]
Let us write $\alpha = \| y \|$ for shorthand and recall $\beta = \| T^{-1} \|^{-1}$. As the first step, we show that the sequence $\{ x_k \}_{k=0}^\infty$ is bounded. By construction in (\ref{eqn::seq}), we bound $\| x_{k+1}\|$ using $\| x_k \|$:
\begin{equation*}
\|x_{k+1} \| \le \| T^{-1} \|( \|y\| + \| \varphi(x_k)\|) \le \frac{\alpha}{\beta} + \frac{\eta}{\beta}\|x_k\|^2.
\end{equation*}
We use this inequality to derive an upper bound on $\{x_k\}$ for all $k$. We define $\xi_0 = 0$ and
\begin{equation*}
\xi_{k+1} = \frac{\alpha}{\beta} + \frac{\eta}{\beta} \xi_k^2, \quad k \ge 0,
\end{equation*}
then clearly $\|x_k\| \le \xi_k$ (which can be shown by induction). It is easy to check (by induction) that the sequence $\{ \xi_k\}_{k=1}^\infty$ is  increasing. Moreover, since $4 \alpha \eta < \beta^2$, the quadratic function
\begin{equation*}
\phi(\xi) = \frac{\alpha}{\beta} + \frac{\eta}{\beta}\xi^2,
\end{equation*}
has two fixed points (namely solutions to $\phi(\xi) = \xi$), and the smaller one satisfies
\begin{equation*}
\xi_{\star} = \frac{2\alpha}{\beta+\sqrt{\beta^2 - 4\eta\alpha}} < \frac{2\alpha}{\beta}.
\end{equation*}
If $\xi_k < \xi_{\star}$, then $ \xi_{k+1} = \phi(\xi_k) \le \phi(\xi_{\star}) = \xi_{\star}$. Thus, by induction, all $\xi_k$ are bounded by $\xi_{\star}$. This implies $\| x_k \| \le \xi_{\star} < 2\alpha / \beta$. The next step is to show that the sequence $\{ x_k \}$ converges. Using the recursive definition (\ref{eqn::seq}) again, we derive
\begin{align*}
\| x_{k+1} - x_k \| &\le \| T^{-1} \| \| \varphi(x_k) - \varphi(x_{k-1}) \| \\
&\le 2\beta^{-1}\eta \max\{ \|x_k\|, \|x_{k-1}\| \}\| x_k - x_{k-1} \| \\
&\le \frac{4\alpha \eta}{\beta^2} \| x_k - x_{k-1} \|.
\end{align*}
Since $4 \alpha \eta / \beta^2 < 1$, the sequence $\{ x_k \}_{k=0}^\infty$ is a Cauchy sequence, and convergence is secured. Let $x^\star \in \mathcal{B}$ be the limit. It is clear by assumption that $x_k \in \mathcal{B}_0$ implies $x_{k+1} \in \mathcal{B}_0$, so $x^\star \in \mathcal{B}_0$ and $\| x^\star \| \le 2\alpha / \beta$ by continuity.

The final step is to show $x^\star$ is a solution to $Tx = y + \phi(x)$. Because $ \{ x_k \}_{k=0}^\infty$ is bounded and $\phi$ satisfies (\ref{ineqn::phi}), the sequence $\{\phi(x_k)\}_{k=0}^\infty$ converges to $\phi(x^\star)$ by continuity and compactness. The linear operator $T$ is also continuous, so we can take limits on both sides of $Tx_{k+1} = y + \phi(x_k)$, we conclude that $x^\star$ is a solution to $Tx = y + \phi(x)$.
\end{proof}

With all the preparations, we are now ready to present the key lemma. As discussed in Section \ref{sec::org}, we set
\begin{equation*}
\mathcal{B}_0 := \{ \overline{Q} \in \mathbb{R}^{d \times r}: \overline{Q} = V_\bot Q \text{ for some } Q \in \mathbb{R}^{(d-r)\times r} \}.
\end{equation*}
which is a subspace of $\mathcal{B} = \mathbb{R}^{d \times r}$. Consider the matrix max-norm $\| \cdot \|_{\max}$ in $\mathcal{B}$.

\begin{lem} \label{lem:main1}
Suppose $|\lambda_r| - \varepsilon > 4r \mu(\tau+2r\kappa)$. Then there exists a solution $\overline{Q} \in \mathcal{B}_0$ to the equation (\ref{eqn:2}) with
\begin{equation*}
\| \overline{Q} \|_{\max} \le  \frac{8(1+r\mu) \kappa}{\big( |\lambda_r|  - \varepsilon \big) \sqrt{d} }\,.
\end{equation*}
\end{lem}

\begin{proof}[Proof of Lemma \ref{lem:main1}]
We will invoke Lemma \ref{lem::borrow} and apply it to the quadratic equation (\ref{eqn:2}). To do so, we first check the conditions required in Lemma \ref{lem::borrow}.

Let the linear operator $\mathcal{T}$ be $\mathcal{T} \overline{Q} = \overline{Q} \overline{L}_1 - \overline{L}_2 \overline{Q}$. By Lemma \ref{lem::linop}, $\mathcal{T}$ has a bounded inverse, and $\beta := \| \mathcal{T}^{-1} \|_{\max}^{-1}$ is bounded from below:
\begin{equation}\label{ineqn::betalb}
\beta \ge |\lambda_r| - 3r \mu (\tau + r\kappa) - \varepsilon \,.
\end{equation}
Let us define $\varphi$ by $\varphi(\overline{Q}) = \overline{Q} H^T \overline{Q}$. To check the inequalities in (\ref{ineqn::phi}), observe that
\begin{equation*}
\| \varphi(\overline{Q}) \|_{\max} \le rd\| \overline{Q} \|_{\max} \| H \|_{\max} \| \overline{Q} \|_{\max}  \le (1 + r\mu) \kappa r \sqrt{d} \, \| \overline{Q} \|_{\max}^2
\end{equation*}
where we used Lemma \ref{lem::H}. We also observe
\begin{align*}
\| \varphi(\overline{Q}_1) - \varphi(\overline{Q}_2) \|_{\max} & = \| \overline{Q}_1 H^T (\overline{Q}_1 - \overline{Q}_2) + (\overline{Q}_1 - \overline{Q}_2) H^T \overline{Q}_2\|_{\max} \\
& \le rd \| \overline{Q}_1 \|_{\max} \| H \|_{\max} \| \overline{Q}_1 - \overline{Q}_2 \|_{\max} +  rd \| \overline{Q}_1 - \overline{Q}_2 \|_{\max} \| H \|_{\max} \| \overline{Q}_2 \|_{\max} \\
& \le 2(1 + r\mu) \kappa r \sqrt{d}\, \max\{ \| \overline{Q}_1 \|_{\max}, \| \overline{Q}_2 \|_{\max} \} \| \overline{Q}_1 - \overline{Q}_2 \|_{\max}.
\end{align*}
Thus, if we set $\eta = (1 + r\mu) \kappa r \sqrt{d}$, then inequalities required in (\ref{ineqn::phi}) are satisfied. For any $\overline{Q}$ with $\overline{Q} = V_\bot Q \in \mathcal{B}_0$, obviously $\varphi(\overline{Q}) = V_\bot Q H^T \overline{Q}  \in \mathcal{B}_0$. To show $\mathcal{T}^{-1} (\overline{Q}) \in \mathcal{B}_0$, let $Q_0 = \mathcal{T}^{-1} (\overline{Q}) $ and observe that
\begin{equation*}
Q_0 \overline{L}_1 - \overline{L}_2 Q_0 = \overline{Q} \in \mathcal{B}_0.
\end{equation*}
By definition, we know $\overline{L}_2 Q_0= V_\bot (E_{22} + \Lambda_2)V_\bot^T Q_0 \in \mathcal{B}_0$, so we deduce $Q_0 \overline{L}_1 \in \mathcal{B}_0$. Our assumption implies $|\lambda_r| > \kappa r \sqrt{\mu}$, so by Lemma \ref{lem::linop}, the matrix $\overline{L}_1$ is invertible, and thus $Q_0 \in \mathcal{B}_0$. The last condition we check is $4 \eta \| H \|_{\max} < \beta^2 $. By Lemma \ref{lem::H} and (\ref{ineqn::betalb}), this is true if
\begin{equation*}
4(1+r\mu)^2\kappa^2r < \big[ |\lambda_r| - 3r \mu (\tau+r\kappa) - \varepsilon \big]^2.
\end{equation*}
The above inequality holds when $|\lambda_r| > 4r \mu(V)(\tau+2r\kappa) + \varepsilon $. Under this condition, we have, by Lemma \ref{lem::borrow},
\begin{equation*}
\| \overline{Q} \|_{\max} \le \frac{2(1+r\mu) \kappa}{\big( |\lambda_r| - 3r\mu(\tau + r\kappa) - \varepsilon \big) \sqrt{d} } \le \frac{8(1+r\mu) \kappa}{ (|\lambda_r| - \varepsilon) \sqrt{d}},
\end{equation*}
where, the second inequality is due to $3r\mu(\tau + r\kappa) \le 3(|\lambda_r| - \varepsilon)/4$.
\end{proof}

The next lemma is a consequence of Lemma \ref{lem:main1}. We define, as in Lemma \ref{lem::Qbar}, that $\omega = 8(1+r\mu) \kappa /  (|\lambda_r| - \varepsilon)$.
\begin{lem}\label{lem::taylor1}
If $r\omega^2 < 1/2$, then
\begin{equation*}
 \| (I_r + \overline{Q}^T \overline{Q})^{-1/2} - I_r \|_{\max} \le r\omega^2,  \qquad \| (I_r + \overline{Q}^T \overline{Q})^{-1/2} \|_{\max}  \le \frac{3}{2}.
\end{equation*}
\end{lem}
\begin{proof}[Proof of Lemma \ref{lem::taylor1}]
By the triangular inequality, the second inequality is immediate from the first one. To prove the first inequality, suppose the spectral decomposition of $\overline{Q}^T \overline{Q}$ is $\overline{Q}^T \overline{Q} = \overline{U}\, \overline{\Sigma}\, \overline{U}^T$, where $\overline{\Sigma} = \text{diag}\{ \overline{\lambda}_1, \ldots, \overline{\lambda}_r \}$ where $\overline{\lambda}_1 \ge \ldots  \ge \overline{\lambda}_r$, and $\overline{U} = [ \overline{u}_1, \overline{u}_2, \ldots, \overline{u}_r]$ where $\overline{u}_1,\ldots, \overline{u}_r$ are orthonormal vectors in $\mathbb{R}^r$. Since $\overline{Q}^T \overline{Q}$ has nonnegative eigenvalues, we have $\overline{\lambda}_r \ge 0$. Using these notations, we can rewrite the matrix as
\begin{equation*}
(I_r + \overline{Q}^T \overline{Q})^{-1/2} - I_r = \sum_{i=1}^r \big( (1+\overline{\lambda}_i)^{-1/2} - 1 \big) \overline{u}_i \overline{u}_i^T\,.
\end{equation*}
Note that $\overline{\lambda}_1 \le \| \overline{Q}^T \overline{Q} \| \le r d\| \overline{Q} \|_{\max}^2 \le r\omega^2$, which implies $\overline{\lambda}_1 < 1/2$. It is easy to check that $1 + |x| \ge (1+x)^{-1/2} \ge 1 - |x|$ whenever $|x| < 1/2$. From this fact, we know $|(1+\overline{\lambda}_i)^{-1/2} - 1| \le \overline{\lambda}_i \le r\omega^2$. Using Cauchy-Schwarz inequality, we deduce that for any $j,k \in [d]$,
\begin{align*}
\big| [ (I_r + \overline{Q}^T \overline{Q})^{-1/2} - I_r]_{jk} \big| & \le\sum_{i=1}^r \big|(1+\overline{\lambda}_i)^{-1/2} - 1 \big| \cdot \big|  \overline{U}_{ji}\overline{U}_{ki} \big| \\
&\le r \omega^2 \cdot \big( \sum_{i=1}^r \overline{U}_{ji}^2 \big)^{1/2} \big( \sum_{i=1}^r \overline{U}_{ki}^2 \big)^{1/2} \\
&\le r\omega^2.
\end{align*}
This leads to the desired max-norm bound.
\end{proof}

\begin{proof}[{\bf Proof of Lemma \ref{lem::Qbar}}]
The first claim of the lemma (the existence of $\overline{Q}$ and its max-norm bound) follows directly from Lemma \ref{lem:main1}. To prove the second claim, we split $\overline{V} - V$ into two parts:
\begin{align}
\overline{V} - V &= V \left( (I_r + Q^TQ)^{-1/2} - I_r \right) + V_\bot Q(I_r +Q^TQ)^{-1/2} \notag\\
& = V \left( (I_r + \overline{Q}^T\overline{Q})^{-1/2} - I_r \right) + \overline{Q}(I_r +\overline{Q}^T\overline{Q})^{-1/2},
\end{align}
where we used identity $V_\bot^T V_\bot = I_{d-r}$. Note that $r\omega < 1/2$ implies $r\omega^2 = (r\omega)^2/r < 1/(4r) < 1/2$. Thus, we can use Lemma \ref{lem::taylor1} and derive

\begin{equation*}
\Big\| V \Big( (I_r + \overline{Q}^T \overline{Q} )^{-1/2} - I_r \Big) \Big\|_{\max} \le \sqrt{\frac{r^2 \mu}{d}}  \| ( I_r + \overline{Q}^T\overline{Q})^{-1/2} - I_r \|_{\max}\le  \sqrt{\frac{\mu}{d}} \, r^2 \omega^2.
\end{equation*}
where we used Cauchy-Schwarz inequality. Using the above inequality and the bound on $\| Q \|_{\max}$ (namely, the first claim in the lemma),
\begin{align*}
\| \overline{V} - V \|_{\max} &\le  \sqrt{\frac{\mu}{d}} \, r^2 \omega^2  +   r \| \overline{Q} \|_{\max} \| (I_r +\overline{Q}^T\overline{Q})^{-1/2} \|_{\max}   \\
&\le (\sqrt{\mu}\,\omega^2 r^2 + 3\omega r/2) / \sqrt{d}.
\end{align*}
Simplifying the bound using $r \omega \le 1/2$ and a trivial bound $\mu \ge 1$, we obtain (\ref{ineqn::Vbarbound}).
\end{proof}

\begin{proof}[ {\bf Proof of Lemma \ref{lem::match}}]
Using the identity in (\ref{ineqn:equiv}), it follows from Davis-Kahan $\sin \Theta$ theorem \citep{DavKah70} and Weyl's inequality that
\begin{equation*}
d(\widetilde{V}, V) \le \frac{ \|E \|_2}{\delta_r - \| E \|_2},
\end{equation*}
when $\delta_r > \| E \|_2$, where $\delta_r = |\lambda_r| - |\lambda_{r+1}|$. Since $\lambda_{r+1} \le \| A - A_r \|_2 \le  \varepsilon$ and $\| E \|_2 \le \tau$, the condition $|\lambda_r| - \varepsilon > 3 \tau$ implies $\delta_r > 3\| E \|_2$. Hence, we have $d(\widetilde{V}, V)  < 1/2$.  Moreover,
\begin{align*}
d(\overline{V}, V) & = \| \overline{V} \overline{V}^T - V V^T \|_2 \le \|\overline{V}  ( \overline{V} - V)^T \|_2 + \| ( \overline{V} - V) V^T \|_2 \\
&\le 2 \| \overline{V} - V \|_2 \le 2\sqrt{rd}\,  \| \overline{V} - V \|_{\max} \\
&\le 4r^{3/2} \sqrt{\mu} \, \omega,
\end{align*}
where we used a trivial inequality $\|M\|_2 \le \|M\|_F \le \sqrt{rd}\, \|M\|_{\max}$ for any $M \in \mathbb{R}^{d \times r}$. Under the condition $|\lambda_r| - \epsilon > 64(1+r\mu)  r^{3/2} \mu^{1/2} \kappa$, it is easy to check that $4r^{3/2} \sqrt{\mu} \, \omega \le 1/2$. Thus, we obtain $d(\overline{V}, V) < 1/2$. By the triangular inequality,
\begin{equation*}
d(\widetilde{V}, \overline{V}) \le d(\widetilde{V}, V) + d(\overline{V}, V) < 1.
\end{equation*}
Since $(\overline{V}, \overline{V}_\bot)^T \widetilde{A} (\overline{V}, \overline{V}_\bot)$ is a block diagonal matrix, $\mathrm{span}(\overline{V})$ is the same as the subspace spanned by $r$ eigenvectors of $\widetilde{A}$. We claim that $\mathrm{span}(\overline{V}) = \mathrm{span}(\widetilde{v}_1,\ldots, \widetilde{v}_r)$. Otherwise, there exists an eigenvector $u \in \mathrm{span}(\overline{V})$ whose associated eigenvalue is distinct from $\widetilde{\lambda}_1,\ldots,\widetilde{\lambda}_r$ (since $\delta_r > 3\| E \|_2$), and thus $u$ is orthogonal to $\widetilde{v}_1,\ldots, \widetilde{v}_r$. Therefore,
\begin{equation*}
\| (\widetilde{V} \widetilde{V}^T - \overline{V} \overline{V}^T) u \|_2 =  \| \overline{V} \overline{V}^T u \|_2 = \| u \|_2.
\end{equation*}
This implies $d(\widetilde{V}, \overline{V}) \ge 1$, which is a contradiction.
\end{proof}

\begin{proof}[{\bf Proof of Theorem \ref{thm::symindiv}}]
We split $\widetilde{V} - V$ into two parts---see (\ref{eqn::Vdecomp}). In the following, we first obtain a bound on $\| R - I_r \|_{\max}$, which then results in a bound on $\| \widetilde{V} - V \|_{\max}$.

Under the assumption of the theorem, $r\omega < 1/2$, so
\begin{equation} \label{ineqn::neatVbar}
\| \overline{V} \|_{\max} \le
 \| \overline{V} - V \|_{\max}  + \| V \|_{\max} \le ( 2\sqrt{\mu}\, r \omega)/\sqrt{d} + \sqrt{r\mu}/\sqrt{d} \le 2\sqrt{r \mu / d}\,.
\end{equation}
To bound $ \| R- I_r \|_{\max}$, we rewrite $R$ as $ R =  \overline{V}^T \overline{V} R = \overline{V}^T \widetilde{V}$. Expand $\overline{V}$ according to (\ref{def::Vbar}),
\begin{equation*}
R = (I_r + \overline{Q}^T \overline{Q} )^{-1/2} (V + \overline{Q})^T  \widetilde{V}\,.
\end{equation*}
Let us make a few observations: (a) $\| \overline{Q}^T\widetilde{V} \|_{\max} \le \sqrt{d} \| \overline{Q} \|_{\max} \le \omega$ by Cauchy-Schwarz inequality; (b) $\| \widetilde{V}^T V \|_{\max} \le 1$ by Cauchy-Schwarz inequality again; and (c) $\| ( I_r + \overline{Q}^T\overline{Q})^{-1/2} - I_r \|_{\max} \le r \omega^2$ by Lemma \ref{lem::taylor1}. Using these inequalities, we have
\begin{align}
\|R -  (V + \overline{Q})^T  \widetilde{V}\|_{\max} &\le r \| ( I_r + \overline{Q}^T\overline{Q})^{-1/2} - I_r  \|_{\max} \; \| (V + \overline{Q})^T  \widetilde{V} \|_{\max} \notag\\
& \le r^2 \omega^2  ( 1 + \omega)\,. \label{ineqn::Rpart1}
\end{align}
Furthermore, by Davis-Kahn $\sin \Theta$ theorem \citep{DavKah70} and Weyl's inequality, for any $i \in [r]$,
\begin{equation} \label{ineqn::dk}
\sin\theta(v_i, \widetilde{v}_i) = \sqrt{1 - \langle v_i, \widetilde{v}_i \rangle^2} \le \frac{\|E\|_2}{\delta - \| E \|_2}\,.
\end{equation}
when $\delta > \| E \|_2$ ($\delta$ is defined in Theorem \ref{thm::symindiv}). This leads to the bound $\sin\theta(v_i, \widetilde{v}_i) \le 2 \| E \|_2 / \delta$ (which is a simplified bound). This is because when $\delta \ge 2 \|E \|_2$, the bound is implied by (\ref{ineqn::dk}); when $\delta < 2 \|E \|_2$, the bound trivially follows from $\sin\theta(v_i, \widetilde{v}_i)  \le 1$. We obtain, up to sign, for $i \le r$,
\begin{equation}\label{ineqn::dkpost}
\sqrt{1 - \langle v_i, \widetilde{v}_i \rangle} \le \sqrt{1 - \langle v_i, \widetilde{v}_i \rangle^2}  \le   \frac{2 \| E \|_2}{\delta}\,.
\end{equation}
In other words, each diagonal entry of $I_r - V^T \widetilde{V} $, namely $1 - \langle v_i, \widetilde{v}_i \rangle $, is bounded by $4\| E \|_2^2/\delta^2$.  Since $\{ \widetilde{v}_i \}_{i=1}^{r}$ are orthonormal vectors, we have $ 1 - \langle v_i, \widetilde{v}_i \rangle^2 \ge \sum_{i' \neq i} \langle v_{i}, \widetilde{v}_{i'} \rangle^2 \ge   \langle v_{i}, \widetilde{v}_{j} \rangle^2 $ for any $i \ne j$, which leads to bounds on off-diagonal entries of $V^T \widetilde{V} - I_r$. We will combine the two bounds. Note that when $\delta \ge 2\| E \|_2$,
\begin{equation*}
\| V^T \widetilde{V} - I_r \|_{\max} \le \max \big\{ \frac{4\| E \|_2^2}{ \delta^2}, \frac{2\| E \|_2}{\delta} \big\} = \frac{2\| E \|_2}{\delta};
\end{equation*}
and when $\delta < 2\| E \|_2$, $\| V^T \widetilde{V} - I_r \|_{\max}$ is trivially bounded by $1$ (up to sign), which is trivially bounded by $2 \| E \|_2 / \delta$. In either case, we deduce
\begin{equation}\label{ineqn::V^TtildeV}
 \| V^T \widetilde{V} - I_r \|_{\max} \le \frac{2\| E \|_2}{\delta}.
\end{equation}
Using the bounds in (\ref{ineqn::Rpart1}) and (\ref{ineqn::V^TtildeV}) and $\| \overline{Q}^T\widetilde{V} \|_{\max} \le \omega$, we obtain
\begin{align*}
\| R - I_r \|_{\max} &\le \| R - (V + \overline{Q})^T \widetilde{V} \|_{\max} + \| V^T \widetilde{V} - I_r \|_{\max} + \| \overline{Q}^T \widetilde{V} \|_{\max} \notag \\
&\le r^2 \omega^2  ( 1 + \omega)  + \frac{2\| E \|_2}{\delta} + \omega\,.
\end{align*}
We use the inequality $r\omega < 1/2$ to simplify the above bound:
\begin{align}
\| R - I_r \|_{\max} &\le r^2\omega^2(1+\omega) + \omega + 2\| E \|_2 / \delta \le  (\frac{1}{2} + \frac{1}{4} + 1) r \omega + 2\| E \|_2 / \delta \notag\\
& \le 2r \omega + 2\| E \|_2 / \delta   \,. \label{ineqn::Rfinal2}
\end{align}

We are now ready to bound $\|\widetilde V - V\|_{\max}$. In (\ref{eqn::Vdecomp}), we use the bounds (\ref{ineqn::neatVbar}), (\ref{ineqn::Rfinal2}), (\ref{ineqn::Vbarbound}) to obtain
\begin{align*}
\| \widetilde{V} - V \|_{\max} &=  \| \overline{V}(R - I_r) + (\overline{V} -V) \|_{\max} \le r\| \overline{V} \|_{\max} \| R - I_r \|_{\max} + \| \overline{V} - V\|_{\max} \\
& \le 2r\sqrt{r \mu /d} \, ( 2r \omega +  2\| E \|_2 / \delta ) + 2r\sqrt{\mu}\, \omega /\sqrt{d} \\
& \le \frac{ (4r^{5/2}\mu^{1/2} + 2r \mu^{1/2}) \omega}{\sqrt{d}} + \frac{4 r^{3/2} \mu^{1/2} \| E \|_2}{\delta \sqrt{d}} \\
&\le \frac{ 48 (1+r\mu)r^{5/2}\mu^{1/2} \kappa}{(|\lambda_r| - \varepsilon)\sqrt{d}} + \frac{4 r^{3/2} \mu^{1/2} \| E \|_2}{\delta \sqrt{d}}.
\end{align*}
Using a trivial inequality $\kappa \le \sqrt{r\mu}\, \tau$, the above bound leads to
\begin{equation*}
\| \widetilde{V} - V \|_{\max} =  O\Big( \frac{r^4 \mu^2 \tau}{(|\lambda_r| - \varepsilon) \sqrt{d}} + \frac{r^{3/2}\mu^{1/2} \| E \|_2}{\delta \sqrt{d}} \Big).
\end{equation*}
\end{proof}

\section{Proofs for Section \ref{sec::asymm}}
Recall the definitions of $\mu_0$, $\tau_0$, $\kappa_0$ and $\varepsilon_0$ in Section \ref{sec:orgasym}. Similar to the symmetric case, we will use the following easily verifiable inequalities.
\begin{equation}\label{ineqn:easy}
\kappa_0 \le \sqrt{r \mu_0} \, \tau_0, \qquad \| E \|_2 \le \left( \sqrt{d_1/d_2}\, \| E \|_{\infty} \cdot \sqrt{d_2/d_1}\, \| E \|_{1} \right)^{1/2} \le \tau_0.
\end{equation}

\begin{lem}\label{lem::H2}
Parallel to Lemma \ref{lem::H}, we have
\begin{equation*}
\| H^d \|_{w} \le (1 + r \mu_0 )\kappa_0\,,
\end{equation*}
where $\kappa_0 = \sqrt{d_1} \, \|EV\|_{\max} \vee \sqrt{d_2} \| E^T U \|_{\max}$  as defined.
\end{lem}
\begin{proof} [Proof of Lemma \ref{lem::H2}]
Recall $H^d =  V^d_\bot (V^d_\bot)^T E^d V^d = E^d V^d - V^d (V^{d})^{T} E^d V^d$. Note $V^d (V^{d})^{T} = \diag(UU^T, VV^T)$ and $\|UU^T\|_{\max} \le r \mu(U)/d_1$, $\|VV^T\|_{\max} \le r \mu(V)/d_2$. Thus,
\begin{align*}
\| H^d \|_{w} & \le  \|E^dV^d\|_{w} + \|V^d (V^d)^T E^d V^d\|_{w} \\
& \le (1 +  d_1 \|UU^T\|_{\max} \vee d_2 \|VV^T\|_{\max} ) \, \|E^d V^d\|_{w}  \le (1 +  r \mu_0 ) \kappa_0 \,.
\end{align*}
\end{proof}

\begin{lem}\label{lem::linop2}
Parallel to Lemma \ref{lem::linop}, if $\sigma_r > 2\kappa_0 r \sqrt{\mu_0}$, then $\overline{L}_1^d$ is a non-degenerate matrix. Furthermore, we have the following bound
\begin{equation}
\inf_{\|Q_0^d\|_{w} = 1} \| Q_0^d \overline{L}_1^d - \overline{L}_2^d Q_0^d \|_{w} \ge  \sigma_r - 4r\mu_0 (\tau_0 + r\kappa_0) - \varepsilon_0 \,,
\end{equation}
where $Q_0^d \in \mathbb{R}^{(d_1 + d_2) \times 2r}$.
\end{lem}
\begin{proof}[Proof of Lemma \ref{lem::linop2}]
Following similar derivations with Lemma \ref{lem::linop}, we have $\| Q_0^d E_{11}^d \|_{w} \le 2\kappa_0 r \sqrt{\mu_0}$, and for any matrix $Q_0^d \in \mathbb{R}^{(d_1 + d_2) \times 2r}$ with $\| Q_0^d \|_w = 1$,
\begin{align*}
\| \overline{L}_2^d Q_0^d \|_{w} = \| (A^d - A_r^d)Q_0^d + V_\bot^d (V_\bot^d)^T E^d V_\bot^d (V_\bot^d)^TQ_0^d \|_w \le \varepsilon_0 + (1+r\mu_0)(\tau_0 + r \kappa_0).
\end{align*}
This can be checked by expressing $Q_0^d$ as a block matrix and expand the matrix multiplication. In particular, one can verify that (i) $ \| (A^d - A_r^d)Q_0^d \|_w \le \varepsilon_0$; (ii) For any matrix $M$ with $d_1+d_2$ rows, $\|V^d (V^d)^T M \|_w \le r\mu_0 \| M \|_w$; (iii) $\| E^d Q_0^d \|_w \le \tau_0$; (iv) $\| E^d V^d (V^d)^TQ_0^d\|_w \le r\kappa_0$. Moreover, $ \| Q_0^d \Lambda_1^d \|_{w} \ge \sigma_r \|Q_0^d \|_{w} \ge \sigma_r$. Thus,
\begin{align*}
\inf_{\|Q_0^d\|_{w} = 1} \| Q_0^d \overline{L}_1^d - \overline{L}_2^d Q_0^d \|_{w} \ge \sigma_r - 4r\mu_0(\tau_0 + r\kappa_0) - \varepsilon_0\,,
\end{align*}
which is the desired inequality in the lemma. In addition, $\overline{L}_1^d$ is non-degenerate if $\sigma_r > 2 \kappa_0 r \sqrt{\mu_0} > 0$.
\end{proof}

\begin{lem} \label{lem:main2}
Parallel to Lemma \ref{lem:main1}, there is a solution $\overline{Q}^d \in \mathcal{B}_0$ to the system (\ref{eqn:2_2}) such that if $\sigma_r - \varepsilon_0 > 16r\mu_0 (\tau_0 + r\kappa_0)$, then
\begin{equation*}
\| \overline{Q}^d \|_{w} \le  \frac{8(1+r \mu_0 ) \kappa_0}{3( \sigma_r  -  \varepsilon_0) }\,.
\end{equation*}
\end{lem}
\begin{proof}[Proof of Lemma \ref{lem:main2}]
We again invoke Lemma \ref{lem::borrow}.
Let $\mathcal{B}$ be the space $\mathbb{R}^{(d_1+d_2) \times 2r}$ equipped with the weighted max-norm $\| \cdot \|_{w}$.  We also define $\mathcal{B}_0$ as a subspace of $\mathcal{B}$ consisting of matrices of the form $V_\bot^d Q^d$ where $Q^d$ has size $(d_1+d_2-2r) \times 2r$. Let the linear operator $\mathcal{T}^d$ be $\mathcal{T}^d \, \overline{Q}^d := \overline{Q}^d \overline{L}_1^d - \overline{L}_2^d \overline{Q}^d$. First notice from Lemma \ref{lem::linop2}, $\mathcal{T}^d$ is a linear operator with bounded inverse, i.e., $\beta := \| (\mathcal{T}^d)^{-1} \|_{w}^{-1}$ is bounded from below by
\begin{equation*}
\beta \ge \sigma_r - 4r \mu_0 (\tau_0 + r\kappa_0) - \varepsilon_0\,.
\end{equation*}
Let $\varphi$ be a map given by $\varphi(\overline{Q}^d) = \overline{Q}^d (H^d)^T \overline{Q}^d$. Note that $H^d \in \mathcal{B}$. Using the (easily verifiable) inequality
\begin{equation}\label{ineq:multiply}
\|M_1M_2^T M_3 \|_w \le 2r \| M_1 \|_w \| M_2^T M_3 \|_{\max} \le 4r \|M_1\|_w \|M_2\|_w \|M_3\|_w \quad \forall \, M_1,M_2,M_3 \in \mathcal{B},
\end{equation}
we derive, by the bound on $\| H^d \|_w$ (Lemma \ref{lem::H2}), that
\begin{align*}
\| \varphi(\overline{Q}^d) \|_{w} \le 4r \| H^d \|_{w} \| \overline{Q}^d \|_{w}^2 \le 4r(1 + r\mu_0) \kappa_0  \, \| \overline{Q}^d \|_{w}^2  \, .
\end{align*}
Moreover, using the inequality (\ref{ineq:multiply}) and the bound on $\| H^d \|_w$ (Lemma \ref{lem::H2}),
\begin{align*}
\| \varphi(\overline{Q}_1^d) - \varphi(\overline{Q}_2^d) \|_{w} & \le 4r \| \overline{Q}_1^d \|_{w} \| H^d \|_{w} \| \overline{Q}_1^d - \overline{Q}_2^d \|_{w} +  4r \| \overline{Q}_1^d - \overline{Q}_2^d \|_{w} \| H^d \|_{w} \| \overline{Q}_2^d \|_{w} \\
& \le 8r(1 + r\mu_0) \kappa_0 \, \max\{ \| \overline{Q}_1^d \|_{w}, \| \overline{Q}_2^d \|_{w} \} \| \overline{Q}_1^d - \overline{Q}_2^d \|_{w}.
\end{align*}
Thus, we can choose $\eta = 4r (1 + r\mu_0) \kappa_0$, and the condition (\ref{ineqn::phi}) in Lemma \ref{lem::borrow} is satisfied. To ensure $4 \eta \| H^d \|_w < \beta^2$, it suffices to require (again by Lemma \ref{lem::H2}),
\begin{equation*}
16 r(1+r\mu_0)^2\kappa_0^2 < \big[ \sigma_r - 4r \mu_0 (\tau_0+r\kappa_0) - \varepsilon_0 \big]^2.
\end{equation*}
It is easily checkable that the above inequality holds when $\sigma_r - \varepsilon_0 > 16r\mu_0 (\tau_0 + r\kappa_0)$. Under this condition, by Lemma \ref{lem::borrow},
\begin{equation*}
\| \overline{Q}^d \|_{w} \le \frac{2 \| H^d \|_w}{\beta} \le  \frac{2(1+r\mu_0) \kappa_0}{ \sigma_r - 4r\mu (\tau_0 + r\kappa_0) - \varepsilon_0 } \le \frac{2(1+r \mu_0 ) \kappa_0}{ \sigma_r  - \varepsilon_0 - ( \sigma_r  - \varepsilon_0)/4 }  \le \frac{8(1+r \mu_0 ) \kappa_0}{3( \sigma_r  -  \varepsilon_0) }\,,
\end{equation*}
which completes the proof.
\end{proof}

\begin{proof}[{\bf Proof of Lemma \ref{lem::Qbar2}}]
The first claim of the lemma (existence of $\overline{Q}^d$ and its max-norm bound) follows from Lemma \ref{lem:main2}. To prove the second claim, we split $\overline{V}^d - V^d$ into two parts:
\begin{align}
\overline{V}^d - V^d = V^d \left( (I_{2r} + (\overline{Q}^d)^T\,\overline{Q}^d)^{-1/2} - I_{2r} \right) + \overline{Q}^d(I_{2r} +(\overline{Q}^d)^T\,\overline{Q}^d)^{-1/2},
\end{align}
Note $\kappa_0 \le \tau_0 \sqrt{r\mu_0}$ (see (\ref{ineqn:easy})). It can be checked that the condition $\sigma_r - \varepsilon_0 > 16r\mu_0 (\tau_0 + r\kappa_0)$ implies $r \omega_0 < 1/3$.  Since $\| (\overline{Q}^d)^T \overline{Q}^d \|_{\max} \le 2\omega_0^2 $ and $ \| (\overline{Q}^d)^T \overline{Q}^d \|_2 \le 2r \| (\overline{Q}^d)^T \overline{Q}^d \|_{\max} \le 4 r\omega_0^2 < 1/2$, similar to Lemma \ref{lem::taylor1}, we have
\begin{align}
& \| ( I_{2r} + (\overline{Q}^d)^T \overline{Q}^d)^{-1/2} - I_{2r} \|_{\max} \le 4 r \omega_0^2,  \label{ineqn::taylor2}  \\
& \| (I_{2r} +(\overline{Q}^d)^T \overline{Q}^d)^{-1/2} \|_{\max} \le 3/2.  \notag
\end{align}
This yields
\begin{align}
\| \overline{V}^d - V^d \|_{w} & = \|V^d \left( (I_{2r} + (\overline{Q}^d)^T \overline{Q}^d )^{-1/2} - I_{2r} \right)\|_w + \| \overline{Q}^d (I_{2r} + (\overline{Q}^d)^T \overline{Q}^d )^{-1/2}\|_{w} \nonumber \\
& \le  \sqrt{2r^2 \mu_0} \,\, 4r \omega_0^2  +  2r \cdot 3/2 \cdot \| \overline{Q}^d \|_{w} \le 8\sqrt{\mu_0}\, \omega_0^2 r^2 + 3 \, \omega_0 r \\
& \le 8\sqrt{\mu_0}\, \omega_0 r / 3 + 3\sqrt{\mu_0}\, \omega_0 r \le  6 \sqrt{\mu_0} \, r \omega_0 . \label{eqn::Vdecomp1Prime}
\end{align}
\end{proof}

\begin{lem} \label{lem::match2}
Suppose $\sigma_r - \varepsilon_0 > \max \{ 16r\mu_0 (\tau_0 + r \kappa_0), 64r^{3/2}\mu_0^{1/2}(1+r\mu_0)\kappa_0 \}$. Then, there exists an orthogonal matrix $R_V \in \mathbb{R}^{r \times r}$ (or $R_U$) such that the column vectors of $\overline{V}R_V$ (and $\overline{U}R_U$) are the top $r$ right (and left) singular vectors of $\widetilde{A}$.
\end{lem}

\begin{proof}[Proof of Lemma \ref{lem::match2}]
Similar to the proof of Lemma \ref{lem::match}, we will prove $d(\overline{V}, V) < 1/2$ and $d(\widetilde{V}, V) \le 1/2$, which would then imply that $\overline{V}$ and $\widetilde{V}$ are the same only up to an orthogonal transformation. The same is true for $\overline{U}$ and $\widetilde{U}$, and we will leave out its proof.

By Weyl's inequality for singular values (also known as Mirsky's theorem \citep{Mir60}), for any $i$, $|\widetilde{\sigma}_i - \sigma_i| \le \| E \|$. By Wedin's perturbation bounds for singular vectors \citep{Wed72},
\begin{equation*}
d(\widetilde{V}, V) \le \frac{ \| E \|_2}{\sigma_r - \| E \|_2}.
\end{equation*}
Note that $\| E \|_2\le \tau_0$ (see \eqref{ineqn:easy}) Under the assumption in the lemma, clearly $\sigma_r - \varepsilon_0 > 3 \tau_0$, and we have $d(\widetilde{V}, V) \le 1/2$. Moreover, by Lemma \ref{lem::Qbar2}, we have $\| \overline{V}^d - V^d \|_w \le 6 r \omega_0 \sqrt{\mu_0}$. Note that each column vector of $V^d$ and $\overline{V}^d$ are $(d_1+d_2)$-dimensional. Looking at the last $d_2$ dimensions, we have $ \| \overline{V} - V \|_{\max} \le  6 r \omega_0 \sqrt{\mu_0 / d_1}$.
\begin{equation*}
d(\overline{V}, V) \le 2 \| \overline{V} - V \| \le 2\sqrt{rd_1}  \| \overline{V} - V \|_{\max} \le 12r^{3/2} \mu_0^{1/2} \omega_0.
\end{equation*}
Under the assumption of the lemma, $d(\overline{V}, V) \le 1/2$. Therefore, we deduce $d(\widetilde{V}, \overline{V}) = 0$, and conclude that there exists an orthogonal matrix $R_V \in \mathbb{R}^{r \times r}$ such that $\widetilde{V} =  \overline{V} R_V $.

\end{proof}

\begin{proof}[{\bf Proof of Theorem \ref{thm::assymbulk}}]
Lemma \ref{lem::Qbar2}, together with Lemma \ref{lem::match2}, implies Theorem \ref{thm::assymbulk}.
\end{proof}

\begin{proof}[{\bf Proof of Theorem \ref{thm::assymindiv}}]
Similar to the proof of Theorem \ref{thm::symindiv}, we first split the difference $\widetilde{V}^d - V^d$:
\begin{equation}\label{eqn::VdecompPrime}
\widetilde{V}^d - V^d = \overline{V}^d (R^d - I_{2r}) + (\overline{V}^d -V^d).
\end{equation}
To bound the first term, note that under our assumption, $r \omega_0 < 1/3$ (derived in the proof of Lemma \ref{lem::Qbar2}), it is easy to check $ \| \overline{V}^d \|_w \le 3\sqrt{r\mu_0}$.
We rewrite the matrix $R^d$ as
\begin{equation*}
R^d = (I_{2r} + (\overline{Q}^d)^T \overline{Q}^d )^{-1/2} (V^d + \overline{Q}^d)^T  \widetilde{V}^d\,.
\end{equation*}
Notice that $\| (\overline{Q}^d)^T\widetilde{V}^d \|_{\max} \le \sqrt{2} \| \overline{Q}^d \|_{w} \le \sqrt{2} \, \omega_0$, $\| (\widetilde{V}^d)^T V^d \|_{\max} \le 1$ and
\begin{align*}
\|R^d -  (V^d + \overline{Q}^d)^T  \widetilde{V}^d\|_{\max} & \le 2 r \| ( I_{2r} + (\overline{Q}^d)^T\overline{Q}^d)^{-1/2} - I_{2r}  \|_{\max} \; \| (V^d + \overline{Q}^d)^T  \widetilde{V}^d \|_{\max} \\
& \le 8 r^2 \omega_0^2  ( 1 + \sqrt{2}\, \omega_0)\,.
\end{align*}
where we used (\ref{ineqn::taylor2}). Following the same derivations as in the proof of Theorem \ref{thm::symindiv}, and using the (easily verifiable) fact $\| E^d \|_2 = \| E \|_2$, we can bound $\| (V^d)^T \widetilde{V}^d - I_{2r} \|_{\max}$ by $2\|E\|_2/\delta_0$. Thus, using $r \omega_0 \le 1/3$, under $\delta_0 > 2 \| E \|_2$, we have
\begin{align}
\| R^d - I_{2r} \|_{\max} &\le\|R^d -  (V^d + \overline{Q}^d)^T  \widetilde{V}^d\|_{\max} + \| (V^d)^T \widetilde{V}^d - I_{2r} \|_{\max} + \| (\overline{Q}^d)^T \widetilde{V}^d \|_{\max} \notag \\
&\le 8 r^2 \omega_0^2  ( 1 + \sqrt{2}\, \omega_0) + \frac{2\|E\|_2}{\delta_0} + \sqrt{2} \, \omega_0 < 4\sqrt{2} \, r\omega_0 + \frac{2\|E\|_2}{\delta_0}\,. \label{ineqn::Rfinal1Prime}
\end{align}
Finally, in order to bound $\|\widetilde V^d - V^d\|_{w}$, we use (\ref{eqn::Vdecomp1Prime}), (\ref{eqn::VdecompPrime}) and (\ref{ineqn::Rfinal1Prime}), and derive
\begin{align*}
\| \widetilde{V}^d - V^d \|_{w} &=  \| \overline{V}^d(R^d - I_{2r}) + (\overline{V}^d -V^d) \|_{w} \le 2r\| \overline{V}^d \|_{w} \| R^d - I_{2r} \|_{\max} + \| \overline{V}^d - V^d\|_{w} \\
& \le 6r\sqrt{r \mu_0} \Big( 4\sqrt{2} \, r \omega_0 + \frac{2\|E\|_2}{\delta_0}\Big) + 6r\sqrt{\mu_0}\, \omega_0 \le \Big(40 r^{5/2} \,\omega_0 + 12r^{3/2}  \, \frac{\|E\|_2}{\delta_0}\Big) \cdot \sqrt{\mu_0} \\
& \le \frac{107 r^{5/2} \mu_0^{1/2}(1+r\mu_0)\kappa_0}{\sigma_r - \varepsilon_0}+ \frac{12r^{3/2}\mu_0^{1/2}\|E\|_2}{\delta_0} \\
& = O \Big( \frac{r^4\mu_0^2 \tau_0}{\sigma_r - \varepsilon_0} + \frac{r^{3/2} \mu_0^{1/2}\|E\|_2}{\delta_0}  \Big).
\end{align*}
This completes the proof.
\end{proof}

\section{Proofs for Section \ref{sec:app}}

\begin{proof}[ {\bf Proof of Proposition \ref{prop:fm}}]
Note first by Weyl's inequality, $|\lambda_i - \overline\lambda_i| \le \|\Sigma_u\| \le C$. So this implies that $\overline{\lambda}_i = \lambda_i(B^T B) \asymp d$ if and only if $\lambda_i =\lambda_i(\Sigma) \asymp d$ for $i \le r$. And furthermore the eigenvalues of $B^T B / d$ are distinct if and only if $\min_{1 \le i \ne j \le r} |\lambda_i(\Sigma) - \lambda_j(\Sigma)| / \lambda_{j}(\Sigma) > 0$.

To prove the equivalency of bounded $\|B\|_{\max}$ and bounded coherence. We first prove the necessary condition.
Again from Weyl's inequality, $\lambda_i(\Sigma) \le C$ for $i \ge r+1$. If $\mu(V)$ is bounded, $\Sigma_{ii}$ must also be bounded, since $\Sigma_{ii} \le \sum_{j= 1}^r v_{ij}^2 \lambda_j(\Sigma) + \lambda_{r+1}(\Sigma) \le C ( \mu(V) + 1)$. Therefore $\|b_i\|^2 \le \|b_i\|^2 + (\Sigma_u)_{ii} = \Sigma_{ii}$ implies $\|B\|_{\max}$ is bounded.  Namely, the factors are pervasive.

On the contrary, if pervasiveness holds, we need to prove that $\mu(V)$ is bounded. Let $B = (\widetilde b_1, \dots, \widetilde b_r)$. Obviously $\overline{\lambda}_i = \|\widetilde b_i\|^2 \asymp d$ and $\overline{v}_i = \widetilde b_i/\|\widetilde b_i\|$. Without loss of generality, assume $\bar\lambda_i$'s are decreasing. So $\|\overline{v}_i\|_{\infty} \le \|B\|_{\max}/\|\widetilde b_i\| \le C/\sqrt{d}$ and $\mu(\overline V) \le C$ where $\overline V = (\overline{v}_1, \dots, \overline{v}_r)$.
By  Theorem \ref{thm::symindiv},
$$
\|\overline{v}_i - v_i\|_{\infty} \le C \frac{\|\Sigma_u\|_{\infty}}{\overline{\gamma}\sqrt{d}}\,,
$$
where $\overline{\gamma} = \min\{\overline{\lambda}_i - \overline{\lambda}_{i+1} : 1 \le i \le r\} \asymp d$ with the convention $\overline{\lambda}_{r +1}=0$. Hence, we have $\|v_i\|_{\infty} \le C/\sqrt{d}$, which implies bounded coherence $\mu(V)$.
\end{proof}

\bibliographystyle{ims}
\bibliography{Reference}

\end{document}